\documentclass[final, nomarks]{dmtcs-episciences}

\usepackage[utf8]{inputenc}
\usepackage[T1]{fontenc}
\usepackage{amsmath,amssymb,amsthm}

\DeclareGraphicsExtensions{.pdf,.png,.eps}
\graphicspath{{figs/}} 

\usepackage[capitalise,noabbrev,nameinlink]{cleveref}

\usepackage{enumitem}

\newtheorem{theorem}{Theorem}[section]
\newtheorem{lemma}[theorem]{Lemma}
\newtheorem{proposition}[theorem]{Proposition}
\newtheorem{corollary}[theorem]{Corollary}

\usepackage{xspace}

\def\root{\ensuremath{\rho}}
\def\trees{\ensuremath{\mathcal{T}_n}} 
\def\nets{\ensuremath{\mathcal{N}_n}}
\newcommand{\PR}{\textup{PR}\xspace}
\newcommand{\PRZ}{\textup{PR$^0$}\xspace}
\newcommand{\PRP}{\textup{PR$^+$}\xspace}
\newcommand{\PRM}{\textup{PR$^-$}\xspace}
\newcommand{\SNPR}{\textup{SNPR}\xspace}
\newcommand{\SNPRZ}{\textup{SNPR$^0$}\xspace}
\newcommand{\SNPRP}{\textup{SNPR$^+$}\xspace}
\newcommand{\rSPR}{\textup{rSPR}\xspace}
\DeclareMathOperator{\dAD}{\ensuremath{d_{\text{\textup{AD}}}}}
\DeclareMathOperator{\dPR}{\ensuremath{d_{\PR}}}
\DeclareMathOperator{\dSNPR}{\ensuremath{d_{\SNPR}}}

\usepackage[round]{natbib}
\newcommand{\doi}[1]{\href{https://dx.doi.org/#1}{\texttt{doi:#1}}}

\author{Jonathan Klawitter\thanks{Supported by the New Zealand Marsden Fund.}}
\title[The agreement distance of rooted phylogenetic networks]{The agreement distance of rooted phylogenetic networks}
\affiliation{School of Computer Science, University of Auckland, New Zealand}
\received{2018-6-18}
\revised{2018-11-27, 2019-2-7, 2019-5-14}
\accepted{2019-5-14}

\begin{document}
\publicationdetails{21}{2019}{3}{19}{4931}
\maketitle
\pdfbookmark[1]{Abstract}{Abstract} 
\begin{abstract}
The minimal number of rooted subtree prune and regraft (rSPR) operations needed to transform one
phylogenetic tree into another one induces a metric on phylogenetic trees -- the rSPR-distance.
The rSPR-distance between two phylogenetic trees $T$ and $T'$ can be characterised by a maximum
agreement forest; a forest with a minimum number of components that covers both $T$ and $T'$.
The rSPR operation has recently been generalised to phylogenetic networks with, among others, the
subnetwork prune and regraft (SNPR) operation.
Here, we introduce maximum agreement graphs as an explicit representations of differences of
two phylogenetic networks, thus generalising maximum agreement forests.
We show that maximum agreement graphs induce a metric on phylogenetic networks -- the agreement distance.
While this metric does not characterise the distances induced by SNPR and other generalisations of
rSPR, we prove that it still bounds these distances with constant factors.
\end{abstract}

\section{Introduction} \label{sec:introduction}
A \emph{phylogenetic tree} is a tree with its leaves labelled by a set of taxa; for example a set
of organisms, species or languages~\citep{SS03,Dun14}.
Phylogenetic trees are used to visualise and study the inferred evolutionary history of such taxa. 
A \emph{phylogenetic network} is a graph with its leaves labelled by a set of taxa, thus generalising a phylogenetic tree.
While a phylogenetic tree models only bifurcating events, a phylogenetic network can also model reticulation events 
like hybridisation, recombination and horizontal gene transfer~\citep{HRS10}.
The phylogenetic trees and networks considered here are all rooted and binary, i.e., they are
directed acyclic graphs and all their vertices except their roots and leaves have degree three.

A tree rearrangement operation transforms one phylogenetic tree into another via a local graph-based change. 
For example, the \emph{rooted subtree prune and regraft} (\rSPR) operation prunes (cuts) 
a subtree of a phylogenetic tree and then regrafts (attaches) it to an edge of the remaining tree, 
resulting in another phylogenetic tree. See \cref{fig:rSPR:example} for an example.
The minimal number of \rSPR operations needed to transform one phylogenetic tree into another
defines the \rSPR-distance, which is a metric on the set of phylogenetic trees~\citep{BS05}.
The problem of computing the \rSPR-distance of two phylogenetic trees is known to be NP-hard, 
but fixed-parameter tractable in its natural parameter~\citep{BS05}.
Moreover, the rSPR-distance induces the notion of neighbourhoods and thus organises the set of
phylogenetic trees into a space. This is important in local search and MCMC algorithms that compute
optimal phylogenetic trees~\citep{Fel04,StJ17}.

\begin{figure}[htb]
  \centering
  \includegraphics{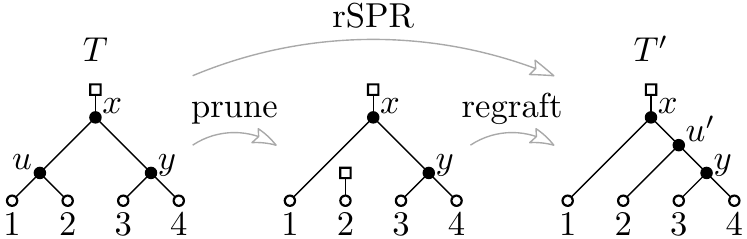}
  \caption{An rSPR operation that transforms $T$ into $T'$ by first pruning the edge $(u, 2)$ at $u$ 
  and then regrafting it to the edge $(x, y)$.
  This process suppresses the vertex $u$ and creates a new vertex $u'$ that subdivides the edge $(x, y)$.}
  \label{fig:rSPR:example}
\end{figure}

An rSPR-sequence that transforms one tree into another tree describes a series of prunings. 
The subtrees unchanged by the sequence form an \emph{agreement forest} for these two trees.
In other words, an agreement forest is the set of trees on which the two trees ``agree'' upon and
that if put together cover each tree.
See \cref{fig:rSPR:sequence} for an example. 
A \emph{maximum agreement forest} (one that has a minimum number of trees) for two phylogenetic trees characterises their rSPR-distance~\citep{BS05}. 
This means that in order to compute or reason about the rSPR-distance it suffices to consider one static structure, 
a maximum agreement forest, instead of a shortest rSPR-sequence.
The notion of maximum agreement forests has proven to be the underpinning concept for almost all
theoretical results as well as practical algorithms that are related to computing the
rSPR-distance~\citep{BS05,BMS08,Wu09,BStJ09,WBZ13,CFS15,BSTW17}. 

\begin{figure}[htb]
  \centering
  \includegraphics{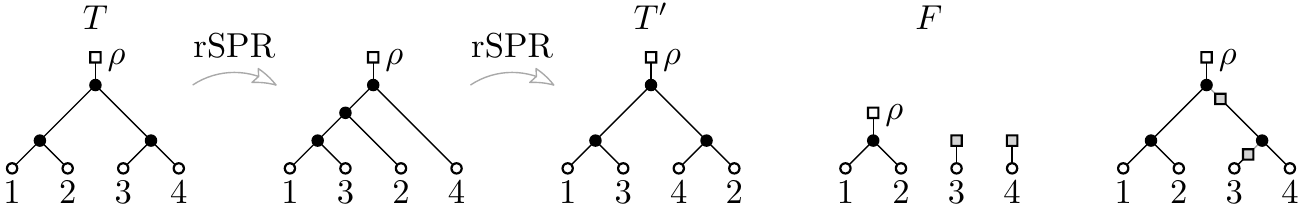}
  \caption{An rSPR-sequence of length two that transforms $T$ into $T'$, an agreement forest $F$ for
  $T$ and $T'$, and on the right how $F$ covers $T$. }
  \label{fig:rSPR:sequence}
\end{figure}

The rSPR operation was recently generalised to network rearrangement operations for phylogenetic networks~\citep{BLS17,GvIJLPS17}.
These generalisations have been studied in terms of computational complexity, shortest sequences and
neighbourhoods~\citep{BLS17,GvIJLPS17,JJEvIS17,FHMW17,Kla17,KL18}.
Like rSPR these operations allow pruning and regrafting of edges. They add extra operations of
adding and removing reticulations (vertices with in-degree two).
It has to be distinguished whether edges can be pruned at their head and tail or only at their tail,
and whether networks can contain parallel edges or not.
Here, we allow parallel edges and consider the \emph{subnetwork prune and regraft} (\SNPR) operation
by \citet{BLS17}, which only allows pruning at the tail, and the \emph{prune and regraft} (\PR) operation that also allows pruning at the head.
Since computing the rSPR-distance between two phylogenetic trees is NP-hard~\citep{BLS17}, it is not
surprising that computing the distance of its generalisations is also NP-hard~\citep{BLS17,GvIJLPS17,JJEvIS17}.
The study of shortest SNPR-sequences and PR-sequences has identified further difficulties for the computation of these distances~\citep{KL18}.
Together with the importance of agreement forests for the rSPR-distance, this motivates the
questions of whether agreement forests are generalisable for rooted binary phylogenetic networks
and, if so, whether they characterise the SNPR or PR-distance.

In this paper, we partially answer these two questions.
First, with \emph{(maximum) agreement graphs} we introduce a generalisation of
(maximum) agreement forests for rooted binary phylogenetic networks (\cref{sec:AG}). 
Then we show that maximum agreement graphs induce a metric on phylogenetic networks, which we call
the \emph{agreement distance} (\cref{sec:AD}).
While maximal agreement forests characterise the rSPR-distance, we show that this agreement
distance does in general not equal the distance induced by a generalisation of rSPR.
On the upside, we prove that the agreement distance bounds the PR-distance and the
SNPR-distance from below naturally and from above with constant factors three
and six, respectively (\cref{sec:bound}).
We end this paper with some concluding remarks (\cref{sec:discussion}).

\section{Preliminaries} \label{sec:preliminaries}
This section contains the definitions of rooted binary phylogenetic networks and trees, of network rearrangement operations and of their induced metrics.
The definition of an agreement graph is given in the next section.

\paragraph{Phylogenetic networks and trees.}
\pdfbookmark[2]{Phylogenetic networks}{PhyNets}
Let $X = \{1, 2, \ldots, n\}$ be a finite set.
A \emph{rooted binary phylogenetic network} $N$ on $X$ is a rooted directed acyclic graph with the
vertices being
\begin{itemize}
  \item the unique \textit{root} labelled $\rho$ with in-degree zero and out-degree one,
  \item \emph{leaves} with in-degree one and out-degree zero bijectively labelled with $X$,
  \item \emph{inner tree vertices} with in-degree one and out-degree two, and
  \item \emph{reticulations} with in-degree two and out-degree one. 
\end{itemize}
The \emph{tree vertices} of $N$ are the union of the inner tree vertices, the leaves and the root. 
An edge $e = (u, v)$ is called a \emph{reticulation edge} if $v$ is a reticulation. 
Following \citet{BLS17}, edges in $N$ can be in \emph{parallel}, that is,
two distinct edges join the same pair of vertices. 
For two vertices $u$ and $v$ in $N$, we say that $u$ is a \emph{parent} of
$v$ and $v$ is a \emph{child} of $u$ if there is an edge $(u, v)$ in $N$.
Similarly, we say that $u$ is an \emph{ancestor} of $v$ and $v$ is a \emph{descendant} of $u$ if
there is a directed path from $u$ to $v$ in $N$. 
An edge $(u, v)$ is an \emph{ancestor} of an edge $(x, y)$ and a vertex $x$ if $v = x$ or if $v$ is an ancestor of
$x$.
In this case, $(x, y)$ is a \emph{descendant} of $(u, v)$ and $v$.
Note that a vertex or an edge cannot be its own ancestor or descendant.

A \emph{rooted binary phylogenetic tree} on $X$ is a rooted binary phylogenetic network that has no reticulations. 
To ease reading, we refer to a rooted binary phylogenetic network (resp. rooted binary phylogenetic
tree) on $X$ simply as a phylogenetic network or network (resp. phylogenetic tree or tree).
Furthermore, let $\nets$ denote the set of all phylogenetic networks on $X$ and let $\trees$ denote
the set of all phylogenetic trees on $X$ where $n = \lvert X \rvert$.

\paragraph{\PR and \SNPR.}
\pdfbookmark[2]{PR and SNPR}{PR}
Let $N \in \nets$ with root $\root$ and let $e = (u, v)$ be an edge of $N$. 
Then the \emph{prune and regraft} (\PR) operation is an operation that transforms $N$ into a
phylogenetic network $N' \in \nets$ in one of the following four ways:

\begin{enumerate}[leftmargin=*,label=(PR$-$)]
  \item[(\PRZ)] If $u$ is a tree vertex (and $u \neq \root$), then delete $e$, suppress $u$,
  subdivide an edge that is not a descendant of $v$ with a new vertex $u'$, and add the edge 
  $(u', v)$; or \\
  if $v$ is a reticulation, then delete $e$, suppress $v$, subdivide an edge that is not an
  ancestor of $u$ with a new vertex $v'$, and add the edge $(u, v')$.
  \item[(\PRP)] Subdivide $(u, v)$ with a new vertex $v'$, subdivide an edge in the resulting
  graph that is not a descendant of $v'$ with a new vertex $u'$, and add the edge $(u', v')$.
  \item[(\PRM)] If $u$ is a tree vertex and $v$ is a reticulation, then delete $e$ and suppress
  $u$ and $v$.
\end{enumerate}

As the name suggests, we understand a \PRZ operation as the process of ``pruning'' the edge $(u, v)$ 
and then ``regrafting'' it to the subdivision of another edge.
We say that the \PRZ \emph{prunes} $(u, v)$ {at} $u$ 
if $u$ is the vertex that gets suppressed and a new vertex $u'$ gets added.
A \PRZ operation that prunes an edge $(u, v)$ at its head vertex $v$ (resp. tail
vertex $u$) is called a \emph{head (tail)} \PRZ operation. 
Note that \PRZ operations do not change the number of reticulations, while \PRM decreases it by
one and \PRP increases it by one. These operations are illustrated in \cref{fig:PR:example}.
Furthermore, note that the  \emph{head} and \emph{tail moves} defined by \citet{GvIJLPS17} 
(and further studied by \citet{JJEvIS17}) conceptually equal head and tail \PRZ, 
but are restricted to networks without parallel edges.

\begin{figure}[htb]
 \centering
 \includegraphics{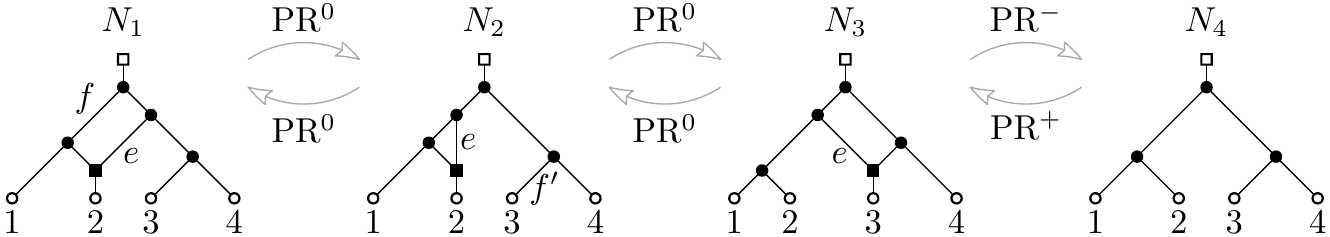}
 \caption{The phylogenetic network $N_2$ (resp. $N_3$) can be obtained from $N_1$ (resp. $N_2$) by the tail \PRZ
 (resp. head \PRZ) that prunes $e$ and regrafts it to $f$ (resp. $f'$).
 The phylogenetic network $N_4$ can be obtained from $N_3$ with the \PRM that removes $e$. 
 Each operation has its corresponding \PRZ and \PRP operation, respectively, that reverses the
 transformation.}
 \label{fig:PR:example}
\end{figure} 

The \emph{subnet prune and regraft} (\SNPR) operation by \citet{BLS17} equals the \PR operation without head \PRZ.
That is, an edge $(u, v)$ can only be pruned at $u$, but not at $v$. 
Further note that the \PR operation restricted to phylogenetic trees is the same as the well known \rSPR operation.
As \citet{BLS17} and \citet{GvIJLPS17} have shown, the different types of \PR operations are all reversible. 
This means that for every \PRZ (or \SNPRZ) that transforms $N$ into $N'$ there exists a \PRZ (resp. \SNPRZ)
that transforms $N'$ into $N$, and that for every \PRP there exists an inverse \PRM.

\paragraph{Distances.}
\pdfbookmark[2]{Distances}{distances} 
Let $N, N' \in \nets$. 
A \emph{\PR-sequence} from $N$ to $N'$ is a sequence 
\begin{displaymath}
	\sigma = (N = N_0, N_1, N_2, \ldots, N_k = N') 
\end{displaymath}
of phylogenetic networks such that, for each $i \in \{ 1, 2, \ldots, k \}$, 
$N_i$ can be obtained from $N_{i-1}$ by a single \PR. The \emph{length} of $\sigma$ is $k$.
The \emph{\PR-distance} $\dPR(N, N')$ between $N$ and $N'$ is the length of a shortest \PR-sequence
from $N$ to $N'$. The \SNPR-distance is defined analogously. 
\citet{BLS17} have shown that the \SNPR-distance is a metric on $\nets$. 
By the definition of the \PR operation and its relation to the \SNPR operation, it thus follows that the \PR-distance is a metric on $\nets$.
Moreover, when restricted to $\trees$, the \SNPR- and the \PR-distance equal the analogously defined
\rSPR-distance~\citep{BLS17}.

\section{Agreement graph} \label{sec:AG}
In this section, we define maximum agreement graphs for two phylogenetic networks $N$
and $N'$. The main idea is to find a graph that can be obtained from both $N$ and $N'$ with a
minimum number of prunings (as defined below).
Throughout this section, let $N, N' \in \nets$ with $r$ and $r'$ reticulations, respectively.
Without loss of generality, assume that $r' \geq r$ and let $l = r' - r$ be the difference of the
number of reticulations of $N$ and $N'$. Furthermore, we assume that a path in a directed graph is
a directed path and contains at least one edge.

\paragraph{Prunings and sprouts.}
\pdfbookmark[2]{Prunings and sprouts}{prunings}
In the last section, we have interpreted the pruning of an edge as half a step of a \PRZ operation (as the name prune and regraft suggests). 
We now extend this to a full operation that yields a new graph.
Let $G$ be a directed graph. Let $u$ be a vertex of $G$ that is either labelled or has degree three. 
Let $(u, v)$ be an edge of $G$.
Then a \emph{pruning} of $(u, v)$ at $u$ is the process of deleting $(u, v)$ 
and adding a new edge $(\bar u, v)$, where $\bar u$ is a new (unlabelled) vertex. 
If $u$ is now an in-degree one, out-degree one vertex, then suppress $u$. 
Note that a pruning does not remove a label from $u$. 
The definition for a pruning of the edge $(u, v)$ at $v$ is analogous.
We mostly apply a pruning to a phylogenetic network or a graph derived from a phylogenetic network.
Therefore, the restriction that $u$ is either labelled or has degree three can be understood as $u$
being either the root $\rho$, a (labelled) leaf, or an internal vertex.

A \emph{sprout} of $G$ is an unlabelled degree one vertex of $G$. 
For example, applying a pruning to a phylogenetic network yields a graph with exactly one sprout.
A \emph{t-sprout} (resp. \emph{h-sprout}) is a sprout that is the tail (resp. head) of its incident edge.

\paragraph{Subdivisions and embeddings.}
\pdfbookmark[2]{Subdivisions and embeddings}{subdivisionEmbedding}
Let $G$ be a directed graph. An edge $(u, v)$ of $G$ is \emph{subdivided} if $(u, v)$ is replaced 
by a path from $u$ to $v$. Recall that we assume that a path is a directed path containing at least one edge.
A \emph{subdivision} $G^*$ of $G$ is a graph that can be obtained from $G$ by subdividing edges of $G$.
Note that if $G$ does not contain in-degree one, out-degree one vertices, 
then there exists a canonical mapping of vertices of $G$ to vertices of $G^*$ and of edges of $G$ to paths of $G^*$.
 
Assume that $G$ is connected. 
Then we say $G$ has an \emph{embedding} into $N$ if there exists a subdivision $G^*$ of $G$ that is a subgraph of $N$.
Now assume that $G$ has components $C_1, \ldots, C_k$. 
Then we say $G$ has an \emph{embedding} into $N$ if the components $C_i$ of $G$, for $i \in \{1, \ldots, k\}$, have embeddings into $N$
to pairwise edge-disjoint subgraphs of $N$. 
Note that these definitions imply that an embedding preserves the direction of edges of $G^*$ into $N$ and maps
a labelled vertex of $G^*$ to a vertex of $N$ with the same label.

\paragraph{Agreement embeddings.}
\pdfbookmark[2]{Agreement embeddings}{agreementEmbedding}
Let $(u, v)$ be an edge of $N$ with $u$ either a labelled vertex, i.e., the root $\rho$, or a degree-three vertex.
Consider a graph $G$ obtained from $N$ by pruning $(u, v)$ at $u$.
Then $G$ has exactly one sprout $\bar u$, and $n + 1$ labelled vertices of which $n$ are bijectively labelled by $X$ and one with $\root$. 
We can distinguish three cases. 
If $u$ is the labelled vertex $\rho$ in $N$, then $G$ contains an isolated labelled $\rho$, say, $\bar u'$.
If $u$ is a reticulation in $N$, then $G$ contains an in-degree two, out-degree zero vertex, say, $\bar u'$.
If $u$ is a inner tree vertex in $N$, then $u$ gets suppressed in the process of the pruning.
In the first two cases, we get a canonical embedding of $G$ into $N$ that is a bijection of the
edges of $G$ to the edges of $N$ and a surjection of the vertices of $G$ to the vertices of $N$.
Only $\bar u$ and $ \bar u'$ of $G$ get mapped to $u$ of $N$. 
In the third case, we obtain such an embedding for a subdivision of $G$ (which reverses the
suppression of $u$) into $N$. 
The case for pruning $(u, v)$ at $v$ is similar. 
Together the three cases motivate the following definition.

Let $G$ be a directed graph.
We say $G$ has an \emph{agreement embedding} into $N$ if there exists an embedding of $G$ into $N$
with the following properties.
\begin{itemize}
  \item An edge $(\bar u, \bar v)$ of $G$ is mapped to a path from a vertex $u$ to a vertex $v$ of $N$ such
  that $\bar u$ is mapped to $u$ and $\bar v$ is mapped to $v$.
  \item The edges of $G$ are mapped to pairwise edge-disjoint paths of $N$ that together cover all
  edges of $N$.
  \item At most two vertices of $G$ are mapped to the same vertex of $N$. In this case, one
  of these two vertices of $G$ is a sprout and the other is either a labelled isolated vertex, or an
  in-degree two, out-degree zero vertex, or an in-degree zero, out-degree two vertex.
  \item For each labelled vertex $v$ of $N$, there exists exactly one vertex $\bar v$ with
  the same label in $G$ and $\bar v$ is mapped to $v$.
\end{itemize}
Note that if $G$ has an agreement embedding into $N$, then $G$ has $n + 1$ labelled vertices of
which $n$ are bijectively labelled by $X$ and one with $\root$.
Furthermore, note that to every inner tree vertex of $N$ either a tree vertex,
or a t-sprout, or an h-sprout and an out-degree two, in-degree zero vertex of $G$ gets mapped.
The situation is similar for reticulations, leafs and the root.

\begin{lemma} \label{clm:AG:embeddingIFFpruning}
Let $G$ be a directed graph and $N \in \nets$.\\
Then $G$ has an agreement embedding into $N$ if and only if $G$ can be obtained from $N$ by a
sequence of prunings.
\end{lemma}
\begin{proof} 
    If $G$ can be obtained from $N$ by a sequence of prunings, then an agreement embedding of $G$
    into $N$ follows naturally.
    So assume that $G$ has an agreement embedding into $N$.
    Then $G$ can be constructed from a sequence of prunings as follows.
    Assume that $G$ contains a t-sprout $\bar u$.
    If $\bar u$ is mapped to a vertex $u$ and the edge $(\bar u, \bar v)$ is mapped to the path from $u$ to $v$ in $N$,
    with $w$ the child of $u$ in this path, then prune the edge $(u, w)$ at $u$.
    This covers either of the cases of when $\bar u$ is mapped to $\rho$, 
    or to the same reticulation to which a degree two vertex of $G$ is mapped,
    or to a tree vertex of $N$ that lies on a path to which an edge of $G$ is mapped.
    In either case, applying this pruning also either creates the isolated vertex $\rho$, a degree two vertex, or suppresses a vertex, respectively.
	A pruning for an h-sprout works analogously.
	We can find such a pruning for each sprout of $G$.
	Now consider the case where we have identified that we want to prune the edge $e = (u, v)$ at $u$ and the edge $f = (u, w)$ at $w$. 
	Let $p$ be the parent of $u$. 
	If we now prune $e$ at $u$, then the edges $f$ and $(p, u)$ are removed when suppressing $u$ and a new edge $f' = (p, w)$ added. 
	In the resulting graph, we cannot prune $f$, but instead now want to prune $f'$ at $w$.
	Further note that since $G$ has an agreement embedding into $N$, no two edges have to be pruned at the same vertex.
	Hence, we can apply one pruning after the other on the edges identified in $N$ or on the edges they get extended to by preceding prunings.   
	As noted, this does not only create the sprouts, but also the labelled, isolated vertices and degree-two vertices 
	and shrinks the path of $N$ to which edges of $G$ get mapped to edges.
	Hence, this sequence of prunings results in $G$.
\end{proof}

\paragraph{Agreement graphs.}
\pdfbookmark[2]{Agreement graphs}{agreementGraph}
Recall that we assume that $N'$ has $l$ more reticulations than $N$.
Let $G$ be a directed graph with connected components $S_1, \ldots, S_k$ and $E_1, \ldots, E_l$,
such that the $E_1, \ldots, E_l$ each consist of a single directed edge.
Then $G$ is an \emph{agreement graph} for $N$ and $N'$ if 
\begin{itemize}
  \item $G$ without $E_1, \ldots, E_l$ has an agreement embedding into $N$, and
  \item $G$ has an agreement embedding into $N'$.
\end{itemize}
For such an agreement graph, we refer to each $S_i$ as an \emph{agreement subgraph} and to each
$E_j$ as a \emph{disagreement edge}.
A \emph{maximum agreement graph} $G$ for $N$ and $N'$ is an agreement graph for
$N$ and $N'$ with a minimal number of sprouts.
\Cref{fig:AD:example:singleton,fig:AD:example:disagreementEdge} give two examples of maximum agreement graphs. 
\begin{figure}[htb]
\centering
  \includegraphics{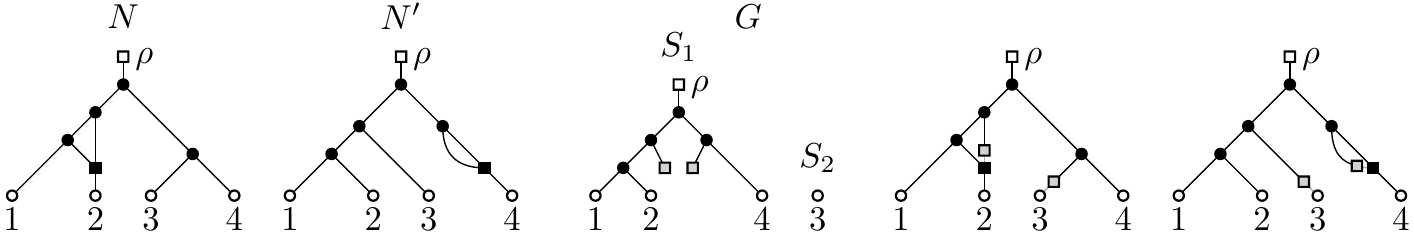} 
  \caption{A maximum agreement graph $G$ for $N$ and $N'$ with its agreement embeddings into $N$ and $N'$ shown on the
  right. Note that the agreement subgraph $S_2$ consists of a labelled isolated vertex.}
  \label{fig:AD:example:singleton}
\end{figure}
\begin{figure}[htb]
\centering
  \includegraphics{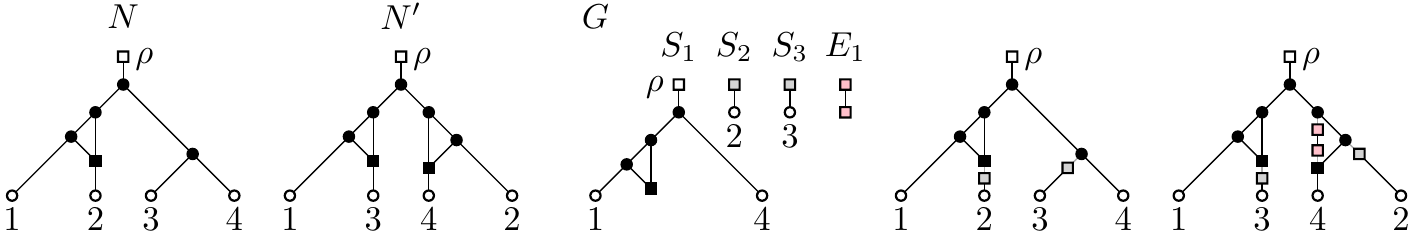}
  \caption{A maximum agreement graph $G$ for $N$ and $N'$ with its agreement embeddings into $N$ and $N'$ shown on
  the right. Note that the disagreement edge $E_1$ is only used for $N'$.}
  \label{fig:AD:example:disagreementEdge}
\end{figure}

An \emph{agreement forest} $F$ for two trees $T$ and $T'$ is an agreement graph for $T$ and $T'$ where
each agreement subgraph $S_i$ for $i \in \{2, \ldots, k\}$ is a phylogenetic tree with an unlabelled root
and $S_1$ is either a phylogenetic tree (with the root labelled $\root$) or an isolated vertex labelled $\root$.
Note that an agreement forest $F$ contains no h-sprouts and 
that thus in the respective agreement embeddings of $F$ into $T$ and $T'$ a sprout of $F$ is mapped
either to the root $\root$ or to a subdivision vertex of an edge of another agreement subgraph.
See again \cref{fig:rSPR:sequence} for an example.
On the other hand, considering shortest \PR-sequences between $N$ and $N'$ in the examples in
\cref{fig:AD:example:singleton,fig:AD:example:disagreementEdge}
shows why in general in an agreement embedding of an agreement graph $G$ a sprout may have to be mapped 
to the same vertex as a labelled isolated vertex ($\rho$ or a leaf of $N$ and $N'$) or a non-suppressible degree-two vertex of $G$. 

Before we show that maximum agreement graphs induce a metric on $\nets$, we establish further
notation and terminology to ease talking about agreement embeddings and agreement graphs.
We use $\bar e, \bar f, \bar u, \bar v$ if we refer to edges or vertices of an agreement graph and
$e, f, u, v$ for edges and vertices of $N$ or $N'$.
If we use symbols like $\bar u$ and $u$ in the same context, then $\bar u$ is usually
mapped to $u$ by the agreement embedding under consideration.

Let $G = (V_G, E_G)$ be a graph with an agreement embedding into a network $N = (V_N, E_N)$.
We say a sprout $\bar u \in V_G$ is \emph{attached} to $\bar e \in E_G$ in $N$ 
	if $\bar u$ is mapped to a vertex $u \in V_N$ that is an internal vertex of the path to which
	$\bar e$ is mapped.
Similarly, we say $\bar u \in V_G$ is \emph{attached} to $\bar x \in V_G$ in $N$ 
	if $\bar u$ and $\bar x$ are mapped to the same vertex $x \in V_N$.
We say an edge $\bar e = (\bar u, \bar v) \in E_G$ is \emph{attached} to $\bar f \in E_G$ in $N$
 	if either $\bar u$ or $\bar v$ is a sprout and attached to $\bar f$.
 	Note that $\bar e$ being attached to $\bar f$ does not imply $\bar f$ being attached to $\bar e$.
Considering the example in \cref{fig:AD:example:singleton} and the agreement embedding of $G$ into
$N$, note that one sprout is attached to the incoming edge of leaf $2$ in $N$ and another sprout is
attached to the isolated vertex labelled $3$ in $N$.

\paragraph{Embedding changes.}
\pdfbookmark[2]{Embedding changes}{niceEmbedding} 
Note that a graph $G$ may have several agreement embeddings into $N$ or $N'$.
We now describe how, in some cases, an agreement embedding can be changed into another one.
For this, let $\bar u$ and $\bar v$ be two t-sprouts of $G$ with outgoing edges $\bar e = (\bar u, \bar w)$ 
and $\bar f = (\bar v, \bar z)$, respectively, such that $\bar u$ is attached to $\bar f$ in $N$.
Let $\bar e$ be mapped to the path $P = (y, \ldots, w)$ in $N$ and let $\bar f$ be mapped to the
path $P' = (x, \ldots, y, \ldots, z)$ in $N$.
Then an \emph{embedding change} of $G$ into $N$ with respect to $\bar u$ and $\bar v$ is the change
of the embedding such that $\bar e$ is mapped to the path $(x, \ldots, y, \ldots, w)$ formed by a
subpath of $P'$ and the path $P$, and such that $\bar f$ is mapped to the subpath $(y, \ldots, z)$ of $P'$. 
See \cref{fig:AD:embeddingChange} for an example. The definition for h-sprouts is analogous.

\begin{figure}[htb]
\centering
  \includegraphics{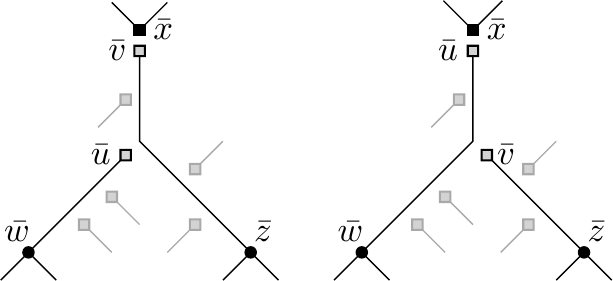}
  \caption{An embedding change with respect to $\bar u$ and $\bar v$.} 
  \label{fig:AD:embeddingChange}
\end{figure}

We now use embedding changes to show that an agreement embedding of $G$ into $N'$ can be changed
into an agreement embedding with some nice properties.
\begin{lemma} \label{clm:AD:specialEmbedding}
Let $N, N' \in \nets$ with $r$ and $r' > r$ reticulations, respectively.
Let $G$ be a maximum agreement graph for $N$ and $N'$.\\
Then there exist an agreement embedding of $G$ into $N'$ such that 
\begin{itemize}
  \item no sprout of an agreement subgraph is attached to a disagreement edge, and
  \item at least one disagreement edge is not attached to any other disagreement edge,
  and
  \item a disagreement edge $E_i$ of $G$ may only be attached to a disagreement edge $E_j$ of $G$ if $j < i$. 
\end{itemize}
\end{lemma} 
\begin{proof}
  Fix an agreement embedding of $G$ into $N'$. 
  Assume that this embedding does not fulfill the first property. 
  Then let $\bar u$ be a sprout of an agreement subgraph of $G$ that is attached to a disagreement edge $(\bar v, \bar w)$ in $N'$. 
  Without loss of generality, assume that $\bar u$ is a t-sprout. 
  Apply an embedding change with respect to $\bar u$ and $\bar v$.
  If $\bar v$ was attached to another disagreement edge $(\bar x, \bar y)$ in $N'$, then repeat this step with $\bar u$ and $\bar x$. 
  Otherwise $\bar u$ is now attached to a vertex or an edge of an agreement subgraph.
  This process terminates since the vertex $u$ to which $\bar u$ gets mapped gets closer to the root in $N'$ with every step.
  Note that the embedding change of $\bar u$ and $\bar v$ may cause a sprout $\bar z$ that was previously attached to $(\bar v, \bar w)$ in $N'$ 
  to now be attached to the edge incident to $\bar u$ in $N'$. 
  However, since this edge is an edge of an agreement subgraph, this does not negatively effect the first property.
  Therefore, every sprout that was previously attached to an edge of an agreement subgraph is still so after each step.
  Hence, each sprout $\bar u$ of an agreement subgraph can be handled after the other and without negatively affecting property three.

  Next, assume that the current embedding fulfills the first, but not the third property.
  Let $E_i = (\bar u, \bar v)$ be a disagreement edge of $G$, starting with $E_i = E_1$. 
  If $\bar u$ and $\bar v$ are each attached to a vertex or an edge of an agreement subgraph or to a disagreement edge $E_j$ with $j < i$ in $N'$,
  then proceed with $E_{i + 1}$.
  Otherwise, without loss of generality, assume that $\bar u$ is attached to a disagreement edge $E_j = (\bar x, \bar y)$ with $j > i$.
  Apply an embedding change with respect to $\bar u$ and $\bar x$.
  The same arguments as above show that eventually $\bar u$ is attached in $N'$ in a good way. 
  Since the embedding change does not affect a sprout of any $E_m$ with $m < i$ or of an agreement subgraph,
  this process does not affect the first property or the previously handled disagreement edges.
  Therefore, each $E_i$ can be handled after the other. 
  Apply analogous steps, if necessary, to $\bar v$ before proceeding with $E_{i + 1}$. 
  The process terminates after $E_i = E_l$ has been handled.
  Finally, note that the third property implies the second.
\end{proof}

Next, we show how to prune a particular edge of $G$ such that the resulting graph is still an agreement graph for $N$ and $N'$. 
\begin{lemma} \label{clm:AD:addPruning}
Let $N, N' \in \nets$. Let $G$ be an agreement graph for $N$ and $N'$. Let $\bar e = (\bar u, \bar v)$ be an edge of $G$.
Then $G$ can be transformed into a graph $G'$ such that
\begin{itemize}
  \item $\bar u$ (or $\bar v$) is a sprout in $G'$, 
  \item $G'$ contains at most one sprout more than $G$, and
  \item $G'$ is an agreement graph for $N$ and $N'$.
\end{itemize} 
\end{lemma}
\begin{proof}
	We prove this for a t-sprout $\bar u$. The proof for an h-sprout works analogously. 
	If $\bar u$ is already a sprout, then there is nothing to do.
	If $\bar u$ is labelled $\root$ or $\bar u$ has degree three, then obtain $G'$ by pruning the edge $\bar e$ at $\bar u$.
	So assume that $\bar u$ is an in-degree zero, out-degree two vertex.
	Consider the agreement embedding of $G$ into $N$. 
	Let $\bar u$ be mapped to $u$ in $N$. 
	Since $u$ has degree three in $N$, there is an h-sprout $\bar w$ mapped to $u$ in $N$.
	This and the following process are illustrated in \cref{fig:AD:addPruning}.
	Then identify $\bar w$ with $\bar u$, i.e., regraft $\bar w$ to $\bar u$, and then prune $\bar e$ from $\bar u$. 
	Let $G''$ be the resulting graph.
	In the agreement embedding of the resulting graph into $N$, the new sprout $\bar u$ is now attached to an edge $\bar f = (\bar x, \bar y)$.
	To get $\bar w$ back, restart this case distinction with the goal to prune $\bar f$ at $\bar y$.
	Note that this process terminates since the number of degree two vertices in $G''$ is one less than in $G$ and thus at some point one of the first two cases has to apply.
	Let $G'$ be the resulting graph when the process has terminated.
	Then $G'$ contains the sprout $\bar u$ with incident edge $\bar e$ and contains at most one sprout more than $G$.
	That is because before the case distinction got restarted, the sprout $\bar w$ got removed first.
	Clearly, $G'$ has an agreement embedding into $N$.
	If one of the first two cases applied, then it is also easy to show that $G'$ has an agreement embedding into $N'$.
	Otherwise, note that in the agreement embedding of $G$ into $N'$ an h-sprout $\bar w'$ is attached to the degree-two vertex $\bar u$.
	For the agreement embedding of $G'$ into $N'$, this sprout $\bar w'$ extends the same way as $\bar w$ got extended in the embedding into $N$ 
	(see again \cref{fig:AD:addPruning}).
\end{proof}

\begin{figure}[htb]
  \centering
  \includegraphics{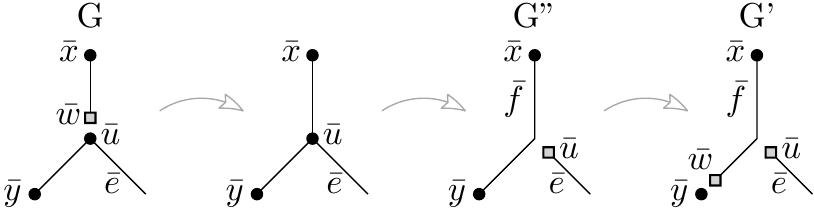}
  \caption{For the proof of \cref{clm:AD:addPruning}, 
  how to prune the edge $\bar e$ at a degree-two vertex $\bar u$.
  First, regraft $\bar w$; second, prune $\bar e$ at $\bar u$; third, reobtain the sprout $\bar w$.} 
  \label{fig:AD:addPruning}
\end{figure}

\section{Agreement distance} \label{sec:AD}
The importance of the notion of maximum agreement forests lies in the fact that it defines a metric
on $\trees$ that equals the \rSPR-distance.
In this section, we show that maximum agreement graphs similarly induce a metric on $\nets$.
Furthermore, we prove that this metric, if restricted to $\trees$, also equals the \rSPR-distance
and is thus NP-hard to compute.

Let $N, N' \in \nets$ and let $l$ be the difference in number of reticulations of $N$ and $N'$.
Let $G$ be a maximum agreement graph for $N$ and $N'$ with $l$ disagreement edges. 
Let $s$ be the total number of sprouts in the agreement subgraphs of $G$.
Then define the \emph{agreement distance $\dAD$} of $N$ and $N'$ as 
\begin{displaymath}
	\dAD(N, N') = s + l\text{.}
\end{displaymath}
This is well defined since $l$ is fixed by $N$ and $N'$, and since $s$ is 
minimal over all agreement graphs for $N$ and $N'$ by the choice of $G$.

\begin{theorem} \label{clm:AD:metric}
The agreement distance $\dAD$ on $\nets$ is a metric.
\end{theorem}
  \begin{proof}
    We have to show that $\dAD$ is symmetric, non-negative, that for all $M, M' \in \nets$ $\dAD(M, M') = 0$ if and only if $M = M'$, 
    and that $\dAD$ satisfies the triangle inequality. Let $N$, $N'$, and $l$ be as above.
    First note that the agreement distance is symmetric and non-negative by definition.
    Second, if  $N = N'$, then $G = N$ is a maximum agreement graph for $N$ and $N'$ with zero    
    sprouts and zero disagreement edges and thus $\dAD(N, N) = 0$.
    Now let $G$ be a maximum agreement graph for $N$ and $N'$ with zero sprouts and zero disagreement edges, i.e., $\dAD(N, N') = 0$.
    Together with the fact that $N$ and $N'$ are internally binary, this implies that every
    unlabelled vertex of $N$ and $N'$, respectively, gets covered by a degree three vertex of $G$.
    Thus $G$ has to consist of a single connected component and has an agreement embedding into
    both $N$ and $N'$ without subdivisions. This in turn implies that $N = G = N'$.
    
    Next, we prove that the agreement distance satisfies the triangle inequality.
    For this let $N, N', N'' \in \nets$ with $r, r'$ and $r''$ reticulations, respectively. 
    With out loss of generality, assume that $r \leq r''$.
    Let $G'$ (resp. $G''$) be a maximum agreement graph for $N$ and $N'$ (resp. $N'$ and $N''$) with 
    $s'$ sprouts in its agreement subgraphs and $l'$ disagreement edges (resp. $s''$ and $l''$).
    For the triangle inequality to hold, we have to show that 
    \begin{displaymath}
    	\dAD(N, N'') \leq d = \dAD(N, N') + \dAD(N', N'') = s' + s'' + l' + l''\text{.} 
    \end{displaymath}
    For this, we construct an agreement graph $G$ for $N$ and $N''$ with $s$ sprouts in its
    agreement subgraphs  and $l$ disagreement edges such that $s + l \leq d$. 
    Note that $G$ does not have to be a maximum agreement graph. 
    Also note that $l$ is fixed by $N$ and $N''$.
    The main idea for the construction of $G$ is to merge $G'$ and $G''$ in terms of the prunings
    they represent in $N$, $N'$ and $N''$.
    Containing, so to say, sprouts from both $G'$ and $G''$ and the right amount of disagreement edges, 
    finding agreement embeddings of $G$ into $N$ and $N''$ will become easy. 
    We first consider the restricted cases of when $N$, $N'$ and $N''$ either have
    the same number of reticulations or only differ in the number of reticulations.

    \noindent\textbf{Case I -- $l' = l'' = 0$.}     
    In this case, by \cref{clm:AG:embeddingIFFpruning} both $G'$ and $G''$ can be obtained from $N'$ by applying $s'$ and $s''$ prunings, respectively.
    We now apply all these prunings to $N'$ to construct $G$ in the following way.
    Like in \cref{clm:AG:embeddingIFFpruning}, we identify to which edges of $N'$ this prunings correspond 
    and whether they prune at the tail or the head of the edge.
    Apply the $s'$ prunings of $G'$ to $N'$ to obtain, of course, $G'$.
    Next, to apply the $s''$ prunings (in $N'$) of $G''$ to $G'$, we have to identify which edges to prune in $G'$.
    
    Assume, without loss of generality, that we want to prune $e = (u, v)$ at $u$ in $N'$. 
   	Further assume $G'$ contains an edge $\bar e = (\bar u, \bar y)$ such that $\bar u$ is mapped to $u$ 
   		and $\bar e$ to a path containing $e$.
	With \cref{clm:AD:addPruning} prune $\bar e$ at $\bar u$ and obtain a graph $\bar G$.
	Note that $\bar G$ has an agreement embedding into $N$ and $N'$. 
	Next, assume $G'$ contains an edge $\bar e' = (\bar x, \bar y)$ such that
		$\bar e'$ is mapped to a path containing $e$ and $u$ as internal vertex.  	
	Then, $G'$ contains a sprout $\bar w$ that is mapped to $u$ (and thus attached to $\bar e'$ in $N'$).
	If $\bar w$ is an h-sprout, prune $\bar e'$ at $\bar x$ with \cref{clm:AD:addPruning} and obtain a graph $\bar G$.
	Note that $\bar G$ has an agreement embedding into $N$ and $N'$. 
	So assume otherwise, namely that $\bar w$ is a t-sprout. 
	Let $\bar w$ have the incident edge $(\bar w, \bar z)$.
	Subdivide $\bar e'$ with a new vertex $\bar u$ and identify $\bar w$ with $\bar u$.
	Prune $\bar e = (\bar u, \bar y)$ at $\bar u$ and 
	then use \cref{clm:AD:addPruning} to prune $(\bar x, \bar z)$ at $\bar x$ to reobtain $\bar w$.
	Let $\bar G$ be the resulting graph.
	Note that $\bar G$ has an agreement embedding into $N'$.
	Furthermore, apply an embedding change with respect to $\bar u$ and $\bar w$ 
	to see that $G$ still has an agreement embedding into $N$.
	Repeat this process (now using $\bar G$ instead of $G'$) for each of the $s''$ sprouts of $G''$.
	Let $G$ be the resulting graph, which by construction has an agreement embedding into $N'$ and $N$.  
	Furthermore, $G$ has at most $s \leq s' + s''$ sprouts.
	
	Lastly, we have to show that $G$ has an agreement embedding into $N''$.
	Consider the agreement embeddings of $G$ and $G''$ into $N'$.
	Let $\bar u$ be a sprout of $G$ obtained for a sprout $\bar u''$ of $G''$.
	If $\bar u$ and $\bar u''$ are mapped to the same vertex $u$ of $N''$,
	then it is straightforward to handle $\bar u$ when obtaining the agreement embedding of $G$ into $N''$.
	On the other hand, $\bar u$ could ``reach beyond'' $u$, that is, its incident edge is mapped to a path containing $u$ as internal vertex.
	This case might be reduced to the former with an embedding change of $G$ into $N'$.
	Otherwise, we know that $\bar u''$ is attached to a degree two vertex $\bar x''$ in $N'$.
	Furthermore, there is then also a sprout $\bar w''$ of $G''$ that is attached to $\bar x$ in the agreement embedding of $G''$ into $N''$.
	Let $\bar w$ be the sprout of $G$ obtained for the sprout $\bar w''$.
	Using the agreement embedding of $G''$ into $N''$ to obtain the agreement embedding of $G$ into $N''$, 
	we then let the sprout $\bar w$ ``reach beyond'' $\bar x''$ in the same way as $\bar u$ does in the agreement embedding of $G$ into $N'$
	(see also \cref{fig:AD:addPruning}).
	To conclude, note that with 
	\begin{displaymath}
	s + l = s \leq s' + s'' = s' + s'' + l' + l''
	\end{displaymath}
	the triangle inequality holds in this case.
    
    \noindent\textbf{Case II.a -- $s' = s'' = 0$ and $r < r' < r''$.}
	In this case, $N'$ can be seen as $N$ plus $l'$ reticulation edges and $N''$ can be seen as $N'$
	plus $l''$ reticulation edges. Thus, $N''$ can also be seen as $N$ plus $l' + l''$ reticulation
	edges. Therefore $G$ consisting of $N$ and $l = l' + l''$ disagreement edges is a desired
	agreement graph for $N$ and $N''$ showing that the triangle inequality holds in this case.
	
    \noindent\textbf{Case II.b -- $s' = s'' = 0$ and $r < r' > r''$.}
    Fix agreement embeddings of $G'$ and $G''$ into $N'$. 
	Colour all edges to which a disagreement edge of $G'$ is mapped orange and to which a disagreement edge of $G''$ is mapped green.
	Intuitively, edges that are now both green and orange in $N'$ are neither in $N$ nor in $N''$. 
	We now align the agreement embeddings of $G'$ (and $G''$) such that a disagreement edge is mapped to either edges
	that are all orange or all green-orange (resp. all green or all green-orange). 
	Note that a disagreement edge is mapped to a path that starts at a tree vertex and ends at a reticulation.
	Furthermore, if such a path contains an internal vertex $v$, then the sprout of another disagreement edge is mapped to $v$. 
	Therefore, to align the agreement embeddings as described above, we can apply a sequence of simple embedding changes to the sprouts of disagreement edges
	as illustrated in \cref{fig:AD:metric:caseIIbAEchange} (i) and (ii) (the rules for h-sprouts and swapped colours are analogous).
	We can further align those disagreement edges of $G'$ and $G''$ that are mapped to green-orange edges with rule (iii) in \cref{fig:AD:metric:caseIIbAEchange}.
	Now let $k'$ be the number disagreement edges of $G'$ (and thus also of $G''$) that are mapped to green-orange edges. 
	   
    \begin{figure}[htb]
	  \centering  
	  \includegraphics{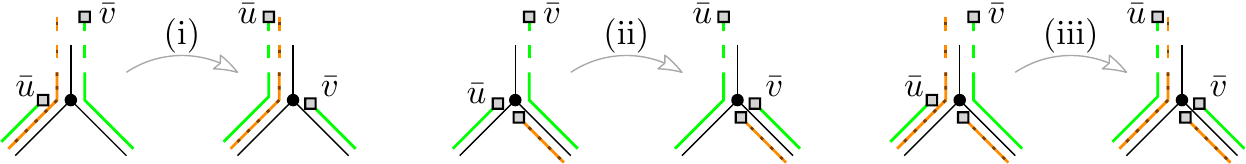}
	  \caption{For Case II.b, embedding changes of $G''$ (green) into $N'$ with respect to $\bar u$ and $\bar v$ 
	  to align the embeddings of disagreement edges of $G'$ (orange with dots) and $G''$ in $N'$.}  
	  \label{fig:AD:metric:caseIIbAEchange}
	\end{figure} 
	
	Obtain a new $N'$ from $N'$ by removing all green-orange edges from $N'$, obtain new $G'$ and $G''$ from $G'$ and $G''$ by removing $k'$ disagreement edges.
	Note that $G''$ has now $k = l'' - k'$ disagreement edges. 
	Clearly, $G'$ (resp. $G''$) has still an agreement embedding into $N$ and $N'$ (resp. $N'$ and $N''$). 
	Then, in $N'$, if a vertex is incident to an uncoloured edge $e$, an orange edge, and a green edge,
     then colour $e$ red. Such a colouring is illustrated in \cref{fig:AD:metric:caseIIb}.
    Next and as long as possible, while a vertex is incident to an uncoloured edge $e$, a red edge and a green or orange edge, colour $e$ red.
    Obtain $S$ from $N'$ by removing all coloured edges and suppressing in-degree one, out-degree one vertices.
    Removing the red edges prevents $S$ from having sprouts.
    Let $G$ be the graph consisting of $S$ and $l$ disagreement edges  
    and $k = l'' - k'$ connected components $F_i$ consisting of a single directed edge.
 	We claim that $G$ is an agreement graph for $N$ and $N''$. 
    
 	\begin{figure}[htb]
	  \centering
	  \includegraphics{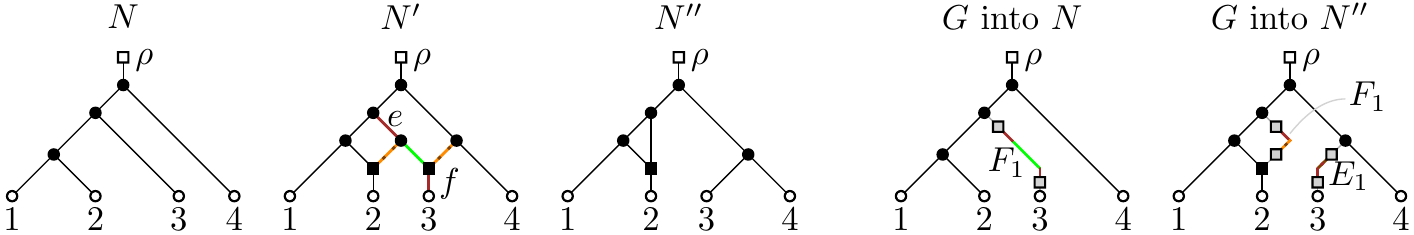}
	  \caption{For Case II.b, $N$ (resp. $N''$) can be obtained from $N'$ by removing the orange (with dots) (resp. green) edges.
	  Embedding $G$ into $N$, the agreement subgraph $F_1$ has to cover not only the green edge, but also the red edges ($e$ and $f$), 
	  which got removed from $N'$ when obtaining $G$ because a disagreement edge of both $N$ and of $N''$ were incident to them.} 
	  \label{fig:AD:metric:caseIIb}
	\end{figure} 
	
 	We construct an agreement embedding of $G$ into $N$.
 	The embedding of $S$ into $N$ is given by the embeddings of $S$ and $N$ into $N'$.
	Let $E_i$ be a disagreement edge of $G''$. 
	Let $P$ be the green path in $N'$ that corresponds to $E_i$. 
	If an edge of $P$ caused the creation of a red edge $e$, extend $P$ by $e$ if possible, i.e., if $P$ would still be a directed path. 
	Next and as long as possible, if $e$ caused another red edge $e'$, extend $P$ by $e'$ if possible.
	Then embed an $F_i$ into $N$ in the way that $P$ is embedded onto $N$ in the embedding of $N$ into $N'$.
	The colours of the edges ensure that this is possible. See again \cref{fig:AD:metric:caseIIb} for an example.
	Furthermore, note that this construction eventually covers all green and red edges.
	Hence, we constructed an agreement embedding of $G$ into $N$.
	Finding an agreement embedding of $G$ into $N''$ works analogously but also uses the disagreement edges of $G$ besides the $F_i$. 
 	Since $l = l' - l''$, we get 
 	\begin{displaymath}
	 	s + l = 2k + l \leq 2l'' + l = l' + l''\text{,} 
 	\end{displaymath}
 	and thus the triangle	inequality also holds in this case.

 	\noindent\textbf{Case II.c -- $s' = s'' = 0$ and $r > r' < r''$.}
	In this case, $N$ and $N''$ can be obtained from $N'$ by adding	$l'$ and $l'' = l + l'$	reticulation edges, respectively.
	Consequently, $N'$ together with $l$ disagreement edges and $l'$ further
	connected components that consists of a single directed edge gives an agreement graph for $N$ and	$N''$. 
	Since $l = l'' - l'$, we get $s + l = 2l' + l = l' + l''$, and thus the triangle inequality	also holds in this case.

	\noindent\textbf{Case III.a -- $r \leq r' \leq r''$.}
	Assume agreement embeddings of $G'$ and $G''$ with nice properties as in \cref{clm:AD:specialEmbedding}.
	We now combine Case I and Case II.b to obtain $G$.
	Let $H$ be the graph $G''$ without its disagreement edges.
	Note that $H$ has an agreement embedding into $N'$ and has $s''$ sprouts.
	Like in Case I, obtain a graph $R$ from $H$ by applying $s'$ prunings in the
	way the $s'$ sprouts of $G'$ are attached to vertices in $N'$. 
	Note that $R$ has an agreement embedding into $N'$ and has at most $s' + s''$ sprouts.
	Then like in Case II.b, obtain a graph $S$ from $R$ by removing all paths from $R$ 
	to which disagreement edges of $G'$ are mapped.
	Again, handle conflicts between a sprout of a disagreement edge of $G'$ and a sprout of $R$	like the red edges in Case II.b.
	Now let $G$ be the graph consisting of $S$ and $l = l' + l''$ disagreement edges.
	Note that $S$ and thus $G$ have at most $s' + s''$ sprouts (ignoring those in the disagreement edges). Hence, $s + l \leq d$.
	Constructing agreement embeddings of $G$ works again by combining the mechanisms from Case I and Case II.b. 
	
	The two cases for when $r \leq r' \geq r''$ and $r \geq r' \leq r''$ can be handled similarly to
	Case III.a together with the ideas from Case II.b and Case II.c. 
	We give a brief outline of how $G$ can be constructed.
	
	\noindent\textbf{Case III.b -- $r \leq r' \geq r''$.}
	Let $S$ be the graph obtained from $N'$ by removing all paths to which the disagreement edges of
	$G'$ and $G''$ are mapped (like in Case II.b) 
	and by applying the prunings of $G'$ and $G''$ in the way they embed into $N'$ (like in Case I).
	Again, in this process we have to take care of cases where two sprouts are mapped to the same vertex. 
	Then the graph $G$ consisting of $S$ and $k \leq l''$ additional directed edges and $l$	disagreement edges 
	is an agreement graph for $N$ and $N''$ with at most $s' + s'' + 2l''$ sprouts
	in agreement subgraphs	and $l = l' - l''$ disagreement edges. 
	Hence, $s + l \leq d$.

	\noindent\textbf{Case III.c -- $r \geq r' \leq r''$.}
	Let $S$ be the graph obtained from $N'$ by applying the prunings of $G'$ and $G''$ in the way they embed into $N'$ (like in Case I). 
	Then the graph $G$ consisting of $S$ and $l'$ additional directed edges and $l$ disagreement edges 
	is an agreement graph for $N$ and $N''$ with at most $s' + s'' + 2l'$ sprouts in agreement subgraphs and $l = l'' - l'$ disagreement edges. 
	Hence, $s + l \leq d$.
	
	This concludes the proof.
  \end{proof}

Next, we show that if we restrict the agreement distance to the space of phylogenetic trees, then it
equals the \rSPR-distance. 

\begin{proposition}\label{clm:AD:treeRestriction}
The agreement distance on $\trees$ is equivalent to the \rSPR-distance.
\end{proposition}
  \begin{proof}
    Let $T, T' \in \trees$. 
    Let $G$ be a maximum agreement graph for $T$ and $T'$ with components $S_1, \ldots, S_m$.
    We distinguish whether $G$ contains an h-sprout or not.
    
    Assume $G$ does not contain an h-sprout. 
    Then $G$ is a maximum agreement forest for $T$ and $T'$.
    Therefore, $\dAD(T, T') = m - 1$, that is, it equals the number of components of $G$ minus one.
    Furthermore, by removing sprouts and their incident edges from $G$ we obtain a forest $F$ 
    that is a maximum agreement forest for $T$ and $T'$ under the definition of \citet{BS05}.
    Hence, the statement follows from Theorem 2.1 by \citet{BS05}.

 	\begin{figure}[htb]
	  \centering
	  \includegraphics{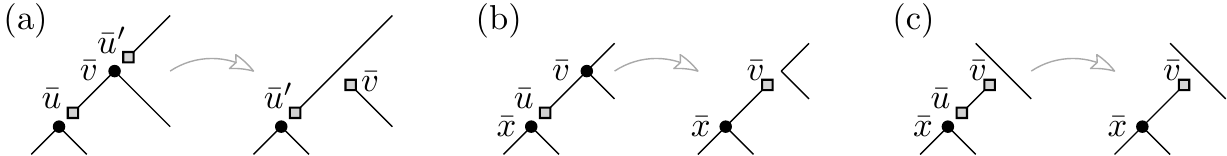}
	  \caption{How to convert h-sprouts from a maximum agreement graph $G$ for two trees to t-sprouts for \cref{clm:AD:treeRestriction},
	  when the h-sprout $\bar u$ is child of a degree two vertex (a), a degree-three vertex (b), or a t-sprout (c), respectively.} 
	  \label{fig:AD:metric:eqSPR}
	\end{figure} 

	Now assume $G$ contains $k$ h-sprouts.
	We now show how to derive a maximum agreement graph $G'$ for $T$ and $T'$ without h-sprouts.
	Assume that $G$ contains an h-sprout $\bar u$ that is a child of a degree two vertex $\bar v$. 
	Note that in the agreement embedding of $G$ into $T$ and $T'$ there is another h-sprout attached to $\bar v$. 
	Thus, deleting $(\bar u, \bar v)$ from $G$ creates a new t-sprout $\bar v$
	such that $G$ is still a maximum agreement graph for $T$ and $T'$  (see \cref{fig:AD:metric:eqSPR} (a)). 
	So assume that $G$ contains no such h-sprout. 
	Hence, $G$ contains $k$ h-sprouts that are adjacent to degree three vertices, to $\rho$ or a t-sprout.
	Then since a tree does not contain reticulations, note that $G$ also contains $k$
	vertices with in-degree zero but out-degree either zero (a labelled leaf of $T$) or two.
	That is because in the agreement embedding of $G$ into $T$ and $T'$ the $k$ h-sprouts have to get mapped to such $k$ vertices. 
	Let $M$ be the set of those vertices. 
	Now, firstly, remove from $G$ the $k$ h-sprouts and their incident edges and suppress resulting	degree two vertices.
	If this results in an unlabelled, isolated vertex, remove it too.
	This does not create any new sprouts since by assumption no h-sprout was incident to a degree two vertex.
	Secondly, add $k$ edges connecting each vertex in $M$ with a new t-sprout (see \cref{fig:AD:metric:eqSPR} (b) and (c)).
	Let $G'$ be the resulting graph. 
	Note that $G'$ contains either the same number of sprouts as $G$ or less if an h-sprout was adjacent to a t-sprout in $G$. 
	(Note that $G$ was actually not a maximum agreement graph if the latter case applies.)
	\Cref{fig:AD:metric:eqSPR} also shows how to derive the agreement embeddings of $G'$ into $T$ and $T'$ from the agreement embeddings of $G$. 
	Hence $G'$ is a maximum agreement graph for $T$ and $T'$ without h-sprouts and the claim follows from the previous case.
  \end{proof}

\citet[Theorem 2.2]{BS05} have shown that computing the \rSPR-distance of two phylogenetic trees is NP-hard. 
Together with \cref{clm:AD:treeRestriction} this implies the following corollary.

\begin{corollary}\label{clm:MAF:NPhard}
Computing the agreement distance is NP-hard.
\end{corollary}

\section{Bounds on rearrangement distances} \label{sec:bound}
After we have shown that the agreement distance equals the \rSPR-distance on \trees, we now consider its relation to the \PR- and \SNPR-distance on \nets.
We start on a positive note concerning the neighbourhoods of a phylogenetic network under \PR and the agreement distance.
\begin{lemma} \label{clm:AD:neighbourhood}
Let $N, N' \in \nets$. Then $\dAD(N, N') = 1$ if and only if $\dPR(N, N') = 1$.
\end{lemma}
  \begin{proof}
	Assume $\dPR(N, N') = 1$. Depending on whether $N'$ can be obtained from $N$ by applying a \PRZ or
	a \PRP operation, obtain a maximum agreement graph $G$ by either mimicking the pruning or adding a
	disagreement edge to $N$. In either case, it follows that $\dAD(N, N') = 1$.
	
	Now assume $\dAD(N, N') = 1$ and let $G$ be a maximum agreement graph for $N$ and $N'$. 
	If $G$ contains a disagreement edge, then it is easy to see that $\dPR(N, N') = 1$. 
	So assume $G$ contains a single sprout $\bar u$. 
	If $\bar u$ is attached to a vertex $\bar x$ of $G$ in the agreement embedding into $N$, then it
	has to be attached to $\bar x$ also in the agreement embedding into $N'$. However, then $N = N'$,
	which is a contradiction to $\dAD(N, N') = 1$.
	If, on the other hand, $\bar u$ is attached to an edge of $G$ in the agreement embedding into $N$
	(and thus into $N'$), then finding a \PRZ that transforms $N$ into $N'$ is straightforward. 
	It follows that $\dPR(N, N') = 1$.
  \end{proof}

Consider the two networks $N$ and $N'$ shown in \cref{fig:AD:unequalToPR}. 
Observe that $\dPR(N, N') = 4$, but that $\dAD(N, N') = 3$ (which can both be shown with an exhaustive search).
Intuitively, the differences arises from the fact that no \PRZ can prune, from $N$ or $N'$, any of the three
sprouts of the shown maximum agreement graph $G$ and regraft it without creating a directed cycle. 
Nor is there a shortest \PR-sequence of length three that uses \PRP and \PRM operations.
This shows that, in general, the agreement distance and the \PR-distance differ on \nets.
Since allowing only tail \PRZ like \SNPR does or not allowing parallel edges increases the distance
in general, it follows that the agreement distances also differs from the \SNPR-distance and
distances of other generalisations of \rSPR.
Furthermore, there exist pairs of phylogenetic networks with $r \geq 1$ reticulations for which
every shortest \PR- or \SNPR-sequence contains a phylogenetic tree~\citep{KL18}.
This implies that along such a sequence reticulation edges get removed and added again. 
Therefore, and even if the \PR-distance (or \SNPR-distance) and the agreement distance would be the
same for such a pair, an agreement graph can in general not fully model every shortest \PR- and \SNPR-sequence.
On the upside, however, we prove now that the agreement distance gives a lower and upper
bound for the \PR-distance with constant factors. We start with the lower bound.

\begin{figure}[htb]
\begin{center}
  \includegraphics{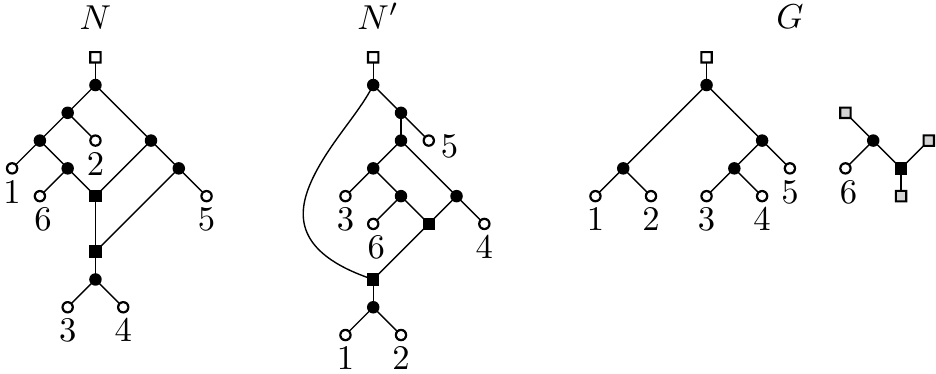}
  \caption{Two phylogenetic networks $N$ and $N'$ with $\dPR(N, N') = 4$, but $\dAD(N, N') = 3$ as
  the maximum agreement graph $G$ shows.}
  \label{fig:AD:unequalToPR}
\end{center}
\end{figure}

\begin{theorem} \label{clm:MAG:PR:lowerBound}
Let $N, N' \in \nets$.
Then $\dAD(N, N') \leq \dPR(N, N')$.
\end{theorem}
  \begin{proof} 
    Given $N$ and $N'$ with \PR-distance $d = \dPR(N, N')$, we construct an agreement graph $G$ of $N$ and $N'$ 
    with $s$ sprouts in the agreement subgraphs and $l$ disagreement edges such that $s + l \leq d$.
    Let $N$ and $N'$ have $r$ and $r'$ reticulations, respectively. 
    Without loss of generality, assume that $r' \geq r$ and let $l = r' - r$.
	The proof is now by induction on $d$.
	If $d = 0$, then $G = N$ is as desired. 
	If $d = 1$, the statement follows from \cref{clm:AD:neighbourhood}.
	Now assume that for each pair of phylogenetic networks $M, M' \in \nets$ with
	\PR-distance at most $d' < d$ for some arbitrary but fixed $d > 1$ there exists an agreement graph
	of $M$ and $M'$ proving that $\dAD(M, M') \leq d'$.

	Fix a \PR-sequence of length $d$ from $N$ to $N'$.
	Let $N'' \in \nets$ be the network of that sequence such that $\dPR(N, N'') = d - 1$ and $\dPR(N'',	N') = 1$. 
	By the induction hypothesis there exists an agreement graph $G'$ for $N$ and $N''$ showing that $\dAD(N, N'') \leq d - 1$.
	We distinguish whether $N'$ is obtained from $N''$ by a \PRZ, a \PRP, or a \PRM operation.

	First, assume that $N'$ can be obtained from $N''$ by pruning the edge $e = (u, v)$ at $u$. 
	Assume $G'$ contains an edge $\bar e = (\bar u, \bar y)$ such that $\bar u$ is mapped to $u$
	and $\bar e$ to a path containing $e$.
	With \cref{clm:AD:addPruning} prune $\bar e$ at $\bar u$ and obtain $G$.
	Then use the agreement embedding of $G$ into $N''$ to obtain an agreement embedding of $G$ into $N'$. 
	Next, assume $G'$ contains an edge $\bar e' = (\bar x, \bar y)$ such that
	$\bar e'$ is mapped to a path containing $e$ and $u$ as internal vertex.  	
	Then, $G'$ contains a t-sprout $\bar w$ that is mapped to $u$ (and thus attached to $\bar e'$ in $N''$).
	The vertex $\bar w$ cannot be an h-sprout, because $u$ is a tree vertex and the previous case does not apply.
	Let $\bar w$ have the incident edge $(\bar w, \bar z)$.
	Subdivide $\bar e'$ with a new vertex $\bar u$ and identify $\bar w$ with $\bar u$.
	Prune $\bar e = (\bar u, \bar y)$ at $\bar u$ and 
	then use \cref{clm:AD:addPruning} to prune $(\bar x, \bar z)$ at $\bar x$ to reobtain $\bar w$.
	Let $G$ be the resulting graph, which has now an agreement embedding into $N'$.
	Considering the embedding of $G$ into $N''$, apply an embedding change with respect to $\bar u$ and $\bar w$
	to see that $G$ still has an agreement embedding into $N$.
	In either case, since $G$ contains at most one sprout more than $G'$, it follows that $\dAD(N, N') \leq \dAD(N,	N'') + 1 \leq d$. 
	The case where $N'$ is obtained from $N''$ by pruning an h-sprout works analogously.

	Second, assume that $N'$ has been obtained from $N''$ by a \PRM that removed the edge $e = (u,	v)$. 
	Note that then $G'$ contains $l + 1$ disagreement edges. 
	Assume $G'$ contains a disagreement edge $E_j = (\bar x, \bar y)$ that maps to a path $P$ 
		that contains $e$ in the agreement embedding of $G'$ into $N''$.
	Note that $u$ is a tree vertex and $v$ a reticulation.
	Therefore, if $P$ contains $u$ as internal vertex, then a t-sprout $\bar w$ is attached to $E_j$ in $N''$ and is mapped to $u$.
	Apply an embedding change with regards to $\bar w$ and $\bar x$.
	Handle the case where $P$ contains $v$ as internal vertex analogously.
	Then $E_j$ is mapped precisely to $e$.	Hence, obtain $G$ from $G'$ by removing $E_j$.
	The agreement embedding of $G$ into $N$	is then the same as of $G'$ and the agreement embedding of $G$ into $N'$ is derived from that of
	$G'$ into $N''$ by removing $E_j$.
	Now assume that $e$ is not covered by a disagreement edge of $G'$. 
	Let $\bar e = (\bar x, \bar y)$ be the edge of $G'$ that covers $e$. 
	With \cref{clm:AD:addPruning} prune $\bar e$ at $\bar x$ and $\bar y$ such that the resulting graph $G''$
	has at most two sprouts more than $G'$ and an agreement embedding into both $N$ and $N''$.
	Consider $\bar e$ now a disagreement edge of $G''$ and consider a disagreement edge of $G''$ an agreement subgraph.
	Then apply the previous case to obtain $G$.
	In either case, $G$ contains one disagreement edge less and at most two sprouts more in its
	agreement subgraphs and therefore $\dAD(N, N') \leq \dAD(N, N'') + 2 - 1 \leq d$.
	
	Lastly, assume $N'$ has been obtained from $N''$ by a \PRP.
	If $l > 0$, obtain $G$ from $G'$ by adding one disagreement edge.
	If $l = 0$, then $G'$ contains one disagreement edge. Thus obtain $G$ from $G'$ by considering this
	disagreement edge an agreement subgraph.
	In either case, it is straightforward to find agreement embeddings of $G$ into $N$ and $N'$.
	Since $G$ contains either one disagreement edge more or two sprouts more but one disagreement edge
	less, it follows again that $\dAD(N, N') \leq d$. This completes the proof.
  \end{proof}

Let $N, N' \in \nets$ with a maximum agreement graph $G = (V_G, E_G)$. Fix agreement embeddings of
$G$ into $N$ and $N'$ and assume that they fulfill the properties of \cref{clm:AD:specialEmbedding}.
In the proof of the upper bound we will construct a \PR-sequence based on agreement embeddings of $G$ along this sequence. 
To ease talking about \PR operations on networks along the sequence based
on vertices and edges of $G$ we define the following terminology. 
Let $\bar u \in V_G$ be a t-sprout with outgoing edge $\bar e = (\bar u, \bar v) \in E_G$.
Let $e = (u, v)$ be the first edge on the path in $N$ to which $\bar e$ is mapped.
\emph{Pruning} $\bar u$ in $N$ then means that the edge $e$ gets pruned at $u$.
\emph{Regrafting} $\bar u$ to an edge $\bar f \in E_G$ in $N$ then means that $e$ gets regrafted to
the edge $f \in E_N$ that is the first edge on the path to which $\bar f$ is mapped.
Let $\bar x$ be a indegree two, outdegree zero vertex or the singleton labelled $\root$ of $G$.
\emph{Regrafting} $\bar u$ to a vertex $\bar x \in V_G$ in $N$ then means that $e$ gets regrafted to
the edge $f \in E_N$ that is the outgoing edge of the vertex $x$ to which $\bar x$ is mapped.
The terminology for h-sprouts is analogously defined.
More precisely, the differences for an h-sprout $\bar u$ are that the edge $\bar e$ is the incoming
edge of $\bar u$, and that $f$ is the last edge of the respective path to which $\bar f$ is mapped
or the incoming edge of the tree vertex $x$.

We say a sprout $\bar u$ is \emph{prunable} (with respect to $N$) if it is attached to an edge
$\bar e$ in $N$ and \emph{unprunable} if it is attached to a vertex $\bar x$ in $N$.
Let $\bar u$ be a sprout that is attached to an edge $\bar f$ (or vertex $\bar x$) in $N'$.
We say the sprout $\bar u$ is \emph{blocked} if regrafting it to $\bar f$ (or $\bar x$) in
$N$ would create a directed cycle; otherwise we call it \emph{unblocked}. 
This implies that there is at least one sprout $\bar v \in V_G$ on the path from $\bar u$ to $\bar
f$ (or $\bar x$) in the embedding of $G$ into $N$. 
We call such a sprout $\bar v$ \emph{blocking}.
See \cref{fig:AD:blockedBlockingAddable} (a) and (b) for examples.

\begin{figure}[htb]
  \centering
  \includegraphics{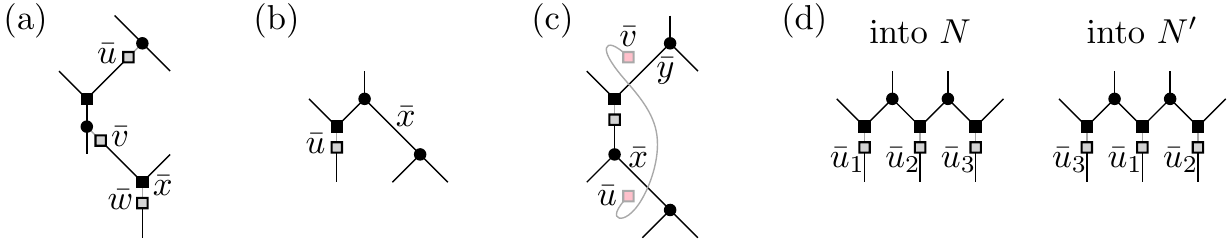}
  \caption{Embeddings of $G$ into $N$ (and $N'$ in (d)). 
  In (a), the sprout $\bar u$ is prunable, but blocked by the blocking sprout $\bar v$ if $\bar u$ is supposed to take the place of $\bar w$. 
  In (b), $\bar u$ is unprunable, but unblocked.
  In (c), the disagreement edge $(\bar u, \bar v)$ is not addable since $\bar y$ is ancestor of $\bar x$.
  In (d), the sprouts $\bar u_1$, $\bar u_2$, and $\bar u_3$ form a replacing cycle.}
  \label{fig:AD:blockedBlockingAddable}
\end{figure}

Let $E_i = (\bar u, \bar v)$ be a disagreement edge and $\bar x$ and $\bar y$ be the vertices or
edges to which $\bar u$ and $\bar v$, respectively, are attached to in $N'$.
If $\bar x$ or $\bar y$ is a disagreement edge $E_j$, then $E_i$ can not be added to $N$ before $E_j$.
Furthermore, if $\bar y$ is an ancestor of $\bar x$ in the embedding into $N$, adding $E_i$ to $N$ would create a directed cycle. 
Therefore we call a disagreement edge $E_i = (\bar u, \bar v)$ \emph{addable} 
if $\bar y$ is not an ancestor of $\bar x$ in $N$ and neither $\bar x$ nor $\bar y$ is a disagreement edge. 
For example, the edge $(\bar u, \bar v)$ in \cref{fig:AD:blockedBlockingAddable} (c) is not addable.

If $\bar u$ is a sprout attached to a vertex $\bar x$ in $N$, then there is a sprout $\bar v$ that
is attached to $\bar x$ in $N'$. We say that $\bar v$ \emph{takes the place} of $\bar u$.
This allows us to define a \emph{replacing sequence} $(\bar u_1, \ldots, \bar u_k)$ of sprouts such
that $\bar u_i$ takes the place of $\bar u_{i + 1}$ with regards to $N$ and $N'$. 
If furthermore $\bar u_k$ takes the place of $\bar u_1$, then we call it a \emph{replacing cycle}.
See \cref{fig:AD:blockedBlockingAddable} (d) for an example.
Note that in a replacing sequence the sprout $\bar u_1$ can be the sprout of a disagreement edge.

\begin{theorem} \label{clm:MAG:PR:upperBound}
Let $N, N' \in \nets$.
Then $\dPR(N, N') \leq 3 \dAD(N, N') $. 
\end{theorem} 
\begin{proof}
    Let $N, N' \in \nets$ with $r$ and $r'$ reticulations, respectively.
    Without loss of generality, assume that $r' \geq r$ and let $l = r' - r$.
    Let $G$ be a maximum agreement graph for $N$ and $N'$.
    Let $S_1, \ldots, S_k$ be the agreement subgraphs of $G$ and $E_1, \ldots, E_l$ be the
    disagreement edges of $G$.
	Fix agreement embeddings of $G$ into $N$ and into $N'$. 
	For the embedding into $N'$, assume that it	fulfills the properties of \cref{clm:AD:specialEmbedding}.
	That is, no sprout of an agreement subgraph $S_i$ is attached to a disagreement edge $E_j$, 
	that at least one disagreement edge (if	one exists) is not attached to any other disagreement edge, 
	and that $E_i$ may be attached to $E_j$ only if $j < i$.

	Let $d = \dAD(N, N')$. 
	To prove the statement we show how to construct a \PR-sequence 
	\begin{displaymath}
		\sigma = (N = N_0, N_1, \ldots, N_m = N') 
	\end{displaymath}
	with $m \leq 3d$.
	While $G$ has an agreement embedding into $N$ and $N'$, 
		it may not have an agreement embedding for several $N_i$, $i \in \{1, \ldots, m - 1\}$.
	However, starting at $N = N_0$, we preserve the mapping of vertices and edges of $G$ to
	vertices and paths of $N_{i-1}$ to $N_{i}$ with each step. 
	Furthermore, along the sequence we map disagreement edges of $G$ to newly added edges.
	In some cases, it is necessary to add edges to $N_{i-1}$ to obtain $N_{i}$ with a \PRP to which
	no disagreement edge will be mapped. We call such edges \emph{shadow edges}.
	From each $N_{i-1}$ to an $N_i$ we only prune edges at a vertex in $N_i$ to which a sprout and
	its incident edge are mapped, or add a disagreement edge, or add or alter a shadow edge.
	We describe any change of $G$, or of the embeddings of $G$ into $N_i$ or $N'$ explicitly. 

	To keep track of the length $m$	of $\sigma$, we credit every \PR operation either to a sprout or to a disagreement edge. 
	When we obtain $N_m = N'$, each sprout and each disagreement edge will have a credit of at most three and, hence, $m \leq 3d$.
	Now, assume $\sigma$ has been constructed up to $N_{i-1}$. 
	
	To obtain $N_i$ we apply the first applicable case of those described below to a sprout or to a disagreement edge.
	Overall the strategy is to first handle easy cases, that is prunable, unblocked sprouts (Case (A) and (A')) and addable disagreement edges (Case (B) and (B')). 
	Then Case (C), (C') and (C'') handle unprunable, unblocked sprouts.
	With Case (D) prunable, blocking sprouts are moved ``aside'' to make them non-blocking 
	and Case (D') adds disagreement edges whose h-sprouts starts a replacing sequence of h-sprouts.
	After exhaustively applying Case (D) and (D'), we can prove that there always exists a prunable sprout (if any sprouts are left).
	A particular sprout (resp. disagreement edge) is subject of at most one application of Case (D) (resp. (D')) and one other case.
	
	\noindent\textbf{(A) Prunable, unblocked sprout to non-shadow edge.}
	If there is a prunable, unblocked sprout $\bar u$ in $N_{i-1}$, then obtain $N_i$ by pruning $\bar
	u$ in $N_{i-1}$ and regrafting it to the edge $\bar f$ or vertex $\bar x$ to which $\bar u$ is
	attached in $N'$. This step gives $\bar u$ a credit of one operation.
	If $\bar u$ is regrafted to a vertex $\bar x$, let $\bar v$ be the sprout that is attached to
	$\bar x$ in $N_i$ (i.e., $\bar u$ takes the place of $\bar v$).
	Apply an embedding change of $G$ into $N$ with respect to $\bar u$ and $\bar v$.
	This whole step is illustrated in \cref{fig:MAG:PR:PrUnbl}.
	Note that $\bar u$ is now attached either to the same edge $\bar f$ or the same vertex $\bar x$	in both $N_i$ and $N'$. 
	Therefore, for the rest of the proof, fix $\bar u$ to $\bar f$ or identify $\bar u$	with $\bar x$, respectively, in $G$. 
	As a result, $\bar u$ with a credit of only one is now not a sprout anymore and thus not subject of another case.
	
	\begin{figure}[htb]
	\centering
	  \includegraphics{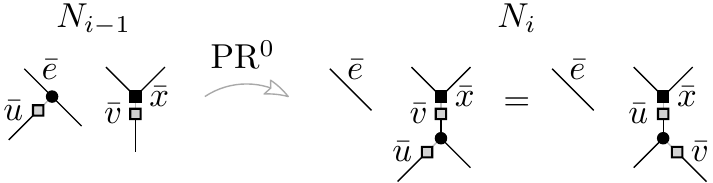}
	  \caption{Illustration of Case (A) where a prunable unblocked sprout $\bar u$  gets regrafted
	  to a vertex $\bar x$, and the subsequent embedding change with regards to $\bar u$ and $\bar v$.}
	  \label{fig:MAG:PR:PrUnbl}
	\end{figure}
	
	\noindent\textbf{(B) Addable disagreement edge without shadow edge.}
	If there exists an addable disagreement edge $E_j$ for $N_{i-1}$, then obtain $N_i$ by adding
	$E_j$ to $N_{i-1}$ with a \PRP. This step gives $E_j$ a credit of one operation. 
	If a sprout of $E_j$ is attached to a vertex in $N'$, then apply again embedding changes of $G$
	into $N_{i}$ like in Case (A). Note that $E_j$ is now attached to the same vertices or edges in
	both $N_i$ and $N'$.
	Therefore, merge the sprouts of $E_j$ with the vertices or edges they are attached to in $G$.
	As a result, $E_j$ with a credit of only one is now no disagreement edge anymore, but an edge of an agreement subgraph $S_{j'}$ of $G$.
	It will therefore not get any further credit.

	\noindent\textbf{(C) Sprout at root, add shadow edge.} 
	If there is an unprunable t-sprout $\bar v$ attached to the root $\root$ in $N_{i-1}$, then there
	is another t-sprout $\bar u$ that is attached to the root in $N'$.
	Assume that $\bar u$ is a sprout of a disagreement edge $(\bar u, \bar w)$ in $N'$, but that Case (B) does not apply. 
	Then $\bar w$ must be attached to another disagreement edge in $N'$.
	This however can be changed with embedding changes (like in \cref{clm:AD:specialEmbedding}) such that
	$(\bar u, \bar w)$ becomes addable and Case (B) applies.
	Therefore assume $\bar u$ is a sprout of an agreement subgraph.
	Since Case (A) does not apply and the root is an ancestor of $\bar u$, it follows
	that $\bar u$ is an unprunable, but unblocked t-sprout in $N_{i-1}$.
	Let $\bar y$ be the in-degree two, out-degree zero vertex to which $\bar u$ is attached in $N_{i-1}$. 
	We now obtain $N_i$ from $N_{i-1}$ by adding and attaching a shadow edge $(\bar w, \bar z)$ 
		from the outgoing edge of $\bar u$ to the incoming edge of leaf $1$ with a \PRP.
	After an embedding change of $G$ into $N_i$ with respect to $\bar w$ and $\bar u$, the sprout $\bar u$ becomes prunable. 
	Give $\bar u$ a credit of one and apply Case (A) to obtain $N_{i+1}$. 
	In total, $\bar u$ gets a credit of two and in $N_{i+1}$ and $N'$ no sprout is attached to the root anymore. 
	This whole step is illustrated in \cref{fig:MAG:PR:Unpr}.
	As mentioned above, the embedding of $G$ into $N_{i+1}$ does not cover all edges anymore, since no
	edge is mapped to the shadow edge.	
	
	\noindent\textbf{(C') Sprout at leaf, add shadow edge.}
	This case is analogous to Case (C) but for h-sprouts. 
	Here, if there is an unprunable h-sprout $\bar v$ attached to a leaf $l$ in $N_{i-1}$, 
	then there is another unprunable, unblocked h-sprout $\bar u$ that takes the place of $\bar v$.
	Then obtain $N_i$ again by adding a shadow edge from the outgoing edge of $\root$ to the incoming edge of $\bar u$. 
	After applying an embedding change, obtain $N_{i+1}$ by	pruning $\bar u$ and attaching it to the incoming edge $\bar f$ of $\bar v$.
	After another embedding change, merge $\bar u$ with the leaf $l$.
	If $l = 1$ and there is a shadow edge $(\bar w, \bar z)$ attached to $\bar f$, then attach $\bar u$ above $\bar z$ to $\bar f$. 
	This way, $\bar z$ is attached to the incoming edge of $l = \bar u$ and not to the incoming edge of $\bar v$ after the embedding change.
	
	\noindent\textbf{(A') Prunable, unblocked sprout to shadow edge.}
	If after the previous two cases, there is again a prunable, unblocked sprout $\bar u$, apply Case (A) again. 
	However, if in this process $\bar u$ gets regrafted to a shadow edge incident to the vertex $\bar x$, 
	then remove the shadow edge with a \PRM after the embedding change of Case (A). 
	This results in a total credit of two for $\bar u$ -- one for the \PRZ to move $\bar u$ and one for the \PRM.
	
	\noindent\textbf{(B') Addable disagreement edge with shadow edge.}
	Similarly, if there is now an addable disagreement edge $E_j = (\bar u, \bar v)$, apply Case (B) in the following way.
	Assume that $\bar v$ of $E_j$ is supposed to get regrafted to a vertex $\bar y$ with an incoming shadow edge $\bar f = (\bar w, \bar z)$.
	Then apply a \PRZ to $N_{i-1}$ to prune $\bar f$ at $\bar w$ and regraft it where $\bar u$ is supposed to be attached. 
	Then again, if $\bar u$ is supposed to be attached to a vertex $\bar x$ with an	outgoing shadow edge $\bar f'$, 
	remove $\bar f'$ with a \PRM operation after an embedding change. 
	This step is illustrated in \cref{fig:MAG:PR:DoubleShadow}.
	The case where only $\bar u$ is supposed to be attached to a vertex with an incident shadow edge 
	but not $\bar v$ is handled analogously. 
	If there is no shadow edge involved for either $\bar u$ or $\bar v$, then Case (B) directly applies.
	In either case, the total credit for $E_j$ is at most two.
	
	\begin{figure}[htb]
	\centering
	  \includegraphics{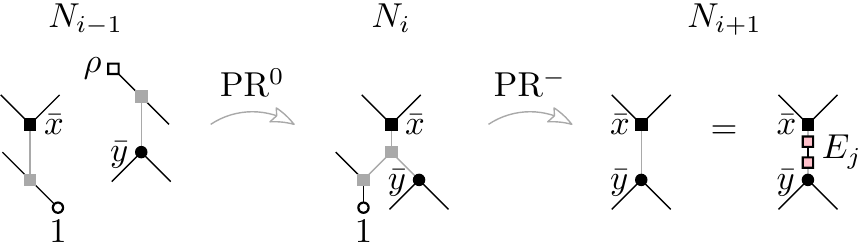}
	  \caption{Illustration of Case (B') with two shadow edges.}
	  \label{fig:MAG:PR:DoubleShadow}
	\end{figure}
	
	The next case is used to decrease the number of blocking sprouts.\\
	\noindent\textbf{(D) Blocked and blocking, but prunable sprout.}
	Let $\bar u$ be a prunable, blocked sprout that is blocking another sprout in $N_{i-1}$.
	Then obtain $N_i$ from $N_{i-1}$ by pruning $\bar u$ and regrafting it to the outgoing edge of $\root$ 
	if $\bar u$ is a t-sprout, or to the incoming edge of leaf 1 otherwise.
	Note that $\bar u$ is now not blocking any other sprout in $N_i$.
	This step gives $\bar u$ a credit of one. 
	Later on, $\bar u$ will get one or two more credit, depending on whether Case (A) or (A') will apply to it. 
	
	\noindent\textbf{(D') Non-addable disagreement edges attached to vertex.}
	Let $E_j = (\bar u, \bar v)$ be an non-addable disagreement edge for which $\bar v$ is attached to a vertex $\bar x$ in $N'$.
	That means that a replacing sequence of h-sprouts starts with $\bar v$ of $E_j$ -- we change this now.
	Obtain $N_i$ from $N_{i-1}$ by adding an edge $(\bar u, \bar v)$ from the outgoing edge of $\root$ to the incoming edge of $\bar x$. 
	Identify $E_j$ with this new edge and then, after an embedding change, merge $\bar v$ with $\bar x$. 
	The vertex $\bar u$ is now a non-blocking and prunable, but blocked t-sprout with a credit of one (just like the sprouts of Case (D)). 
	Note that, after Case (D') does not apply anymore, there can be no replacing sequence of h-sprouts that starts with a sprout of a disagreement edge left. 
	(We do not, maybe even cannot, do the analogous for disagreement edges that start a replacing sequence of t-sprouts.)
	
	Applying Case (A) or Case (A') may now start with a sprout that has already a credit of one.
	However, as in both cases the credit is increased by at most two, the credit will afterwards be at most three.

	So far we have applied Case (A) and (B) until not further possible. 
	Then Case (C) and (C') are applied at most once and $n$ times, respectively.
	We then apply Cases (A), (A'), (B), (B') as long as possible.
	If then applicable we apply Case (D) or (D') and repeat this loop.
	Next, we show that if neither of the previous cases applies but there are still sprouts in
	$N_{i-1}$ that there is then at least one unprunable, unblocked sprout in $N_{i-1}$.
	 
	\noindent\textbf{Existence of unblocked sprout.}
	Assume that there exists a replacing cycle $\tau$ of, without loss of generality, t-sprouts in $N_{i-1}$. 
	Then note that for a t-sprout to be blocked the vertex or edge it will be attached to has to be a descendant. 
	Since phylogenetic networks are acyclic, the sprouts in $\tau$ can not all replace a descendant. 
	Therefore one of the sprouts has to be an unblocked sprout.
	
	Next, assume that there is no replacing cycle in $N_{i-1}$.		
	If no unprunable t-sprout $\bar u$ exists, then the h-sprout with no ancestor h-sprout in $N_{i-1}$ is an unblocked sprout.
	So assume otherwise and let $\bar u$ be an unprunable t-sprout with no descendant t-sprout in $N_{i-1}$.
	If $\bar u$ is unblocked, we are done; so assume otherwise. 
	This means that the vertex or edge to which $\bar u$ is supposed to be regrafted is a descendant of $\bar u$ in $N_{i-1}$.
		Thus, by the choice of $\bar u$, it can only be blocked by an h-sprout $\bar v$.
			Since Case (D) moved prunable, blocking sprouts aside, $\bar v$ has to be unprunable.
			If $\bar v$ is unblocked, we are done; so assume otherwise.
				Then there is a replacing sequence $\tau = (\bar v_1, \ldots, \bar v_m)$ with $\bar v = \bar v_i$ 
				for some $i \in \{2, \ldots, m\}$.
				Note that $\bar v_1$ is prunable since Case (D') does not apply and since there are no replacing
				cycles anymore and thus $\bar v \neq \bar v_1$.
				Since further Case (D) does not apply, $\bar v_1$ is also not a blocking sprout.
				Assuming that there is no unblocked sprout in $\tau$, we know that for every $1 \leq j < i$
				the h-sprouts $\bar v_1$ to $\bar v_j$ are all descendants of $\bar v_{j + 1}$ to $\bar v_i$ and
				thus also of $\bar u$.
	Since $\bar v_1$ is blocked, there has to be an unprunable h-sprout $\bar v'$ blocking $\bar v_1$.
	Note that $\bar v'$ is a descendant of $\bar v_2$ and thus not in $\tau$.
	The situation with $\bar v'$ is now the same as with $\bar v$ 
	and the chain of descendants of h-sprouts below $\bar u$ contains now $\bar v = \bar v_i, \ldots,  \bar v_2, \bar v'$.
		Finally, we either find an unprunable h-sprout in the replacing sequence $\tau' \neq \tau$
		that contains $\bar v'$ or the chain of descendants of h-sprouts below $\bar u$ grows longer with h-sprouts $\bar v'_2$ and $\bar v''$.
		Since $N_{i-1}$ is finite this chain cannot grow indefinitely and thus at some point we find an
		unblocked h-sprout.
	
	\noindent\textbf{(C'') Unprunable, unblocked sprout.}
	If there is an unprunable, unblocked sprout $\bar u$ in $N_{i-1}$ that is attached to the edge $\bar f$ 
	or a vertex $\bar x$ in $N'$ that has no shadow edge attached in $N_{i-1}$, 
	then use the same procedure as in Case (C) or (C') to obtain $N_i$ and then $N_{i+1}$.
	This gives $\bar u$ a credit of two, before it gets merged with $\bar x$ or $\bar f$.
	This step is illustrated in \cref{fig:MAG:PR:Unpr}.
	
	\begin{figure}[htb]
	\centering
	  \includegraphics{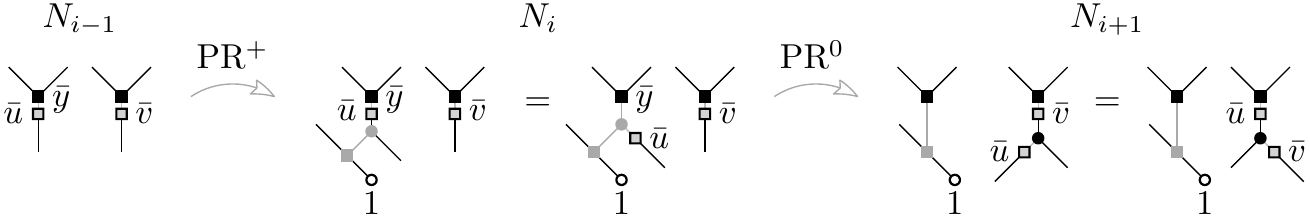}
	  \caption{Illustration of Case (C) and Case (C'') where an unprunable, unblocked sprout $\bar u$
	  is moved to the vertex $\bar x$ with two \PR operations and two embedding changes.}
	  \label{fig:MAG:PR:Unpr}
	\end{figure}
	
	If there is an unprunable, unblocked sprout $\bar u$ in $N_{i-1}$ that is attached to a vertex
	$\bar x$ in $N'$ that has a shadow edge attached in $N_{i-1}$, then apply the process shown in
	\cref{fig:MAG:PR:UnprShadow} to obtain $N_i$ and $N_{i+1}$.
	This gives $\bar u$ a credit of two, before it gets merged with $\bar x$. 
	Note that this moves the shadow edge from $\bar x$ to the vertex to which $\bar u$ was attached to in $N_{i-1}$.

	\begin{figure}[htb]
	\centering
	  \includegraphics{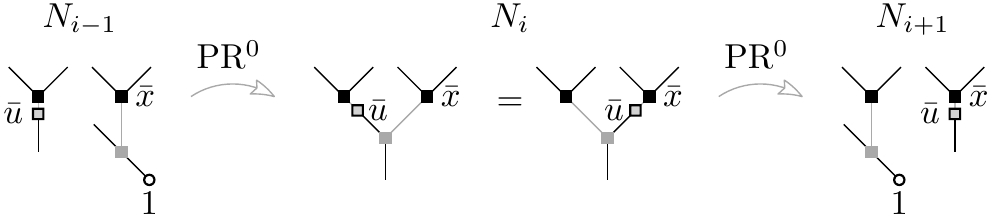}
	  \caption{Illustration of Case (C'') where an unprunable, unblocked sprout $\bar u$ is moved to a
	  vertex $\bar x$ with an incident shadow edge.}
	  \label{fig:MAG:PR:UnprShadow}
	\end{figure}

	Since prunable sprouts cannot block after they got moved aside, since disagreement edges
	cannot block either (by the properties of the agreement embedding into $N'$), 
	and since the number of unprunable sprouts is decreased stepwise,
	the whole process resolves all sprouts and disagreement edges. Hence, $N_m = N'$.
	Since every sprout and every disagreement edge got a credit of at most three, it follows that $m \leq 3d$. 
	This concludes the proof.
  \end{proof}

We prove a relation between the \PR-distance and the \SNPR-distance.

\begin{lemma} \label{clm:MAG:SNPRtoPRrelation}
Let $N, N' \in \nets$. Then 
\begin{displaymath}
	\dPR(N, N') \leq \dSNPR(N, N') \leq 2 \dPR(N, N')\text{.}
\end{displaymath}
\end{lemma}
  \begin{proof}
    The first inequality follows from the definitions of \PR and \SNPR.
    For the second inequality, let $d = \dPR(N, N')$ and $\sigma = (N = N_0, N_1, \ldots, N_d = N')$
    be a \PR-sequence from $N$ to $N'$ of length $d$.
    Then we can construct an \SNPR-sequence $\sigma^* = (N = M_0, M_1, \ldots, M_k = N')$ with $k \leq 2d$ as follows.
    Assume we have constructed $\sigma^*$ up to $M_{j-1} = N_{i-1}$.
    Then, if $N_i$ is obtained from $N_{i-1}$ by a tail \PRZ or a \PRP or a
    \PRM, then apply the same operation to $M_{j-1}$ to obtain $M_{j}$.
    So assume, otherwise; i.e., $N_i$ is obtained from $N_{i-1}$ by a head \PRZ.
    Let $e = (u, v)$ be the edge that gets pruned at $v$ and $f$ be the edge that gets  subdivided to regraft $e$. 
    Obtain $M_i$ from $M_{i-1}$ with the \SNPRP that subdivides $e$ with a new vertex
    $u'$, subdivides $f$ with a new vertex $v'$, and adds the edge $(u', v')$. 
    Next, obtain $M_{i+1}$ from $M_{i}$ by removing $(u', v)$ and suppressing the resulting degree
    two vertices. Then clearly $M_{i+1} = N_i$.
    Since at most two \SNPR operations are needed per \PR, it follows that $k \leq 2d$.
  \end{proof}

The following corollary is a direct consequence of
\cref{clm:MAG:PR:lowerBound,clm:MAG:PR:upperBound,clm:MAG:SNPRtoPRrelation}.

\begin{corollary}\label{clm:MAG:SNPRtoADrelation}
Let $N, N' \in \nets$.
Then 
\begin{displaymath}
	\dAD(N, N') \leq \dSNPR(N, N') \leq 6 \dAD(N, N')\text{.}
\end{displaymath}
\end{corollary}

\section{Concluding remarks} \label{sec:discussion}
In this paper, we defined maximum agreement graphs for two rooted binary phylogenetic networks. 
Like maximum agreement forests for trees, a maximum agreement graph models how the two networks
agree on subgraphs derived from a minimum number of prunings. If the two networks have different
numbers of reticulations, then agreement graphs also model how they disagree on that.
Based on this, we defined the agreement distance on phylogenetic networks. First, we showed that
the agreement distance equals the \rSPR-distance when calculated for phylogenetic trees.
For phylogenetic networks, the agreement distance is a lower bound on the \PR- and \SNPR-distance.
Furthermore, it bounds both the \PR- and \SNPR-distance from above by a factor of at most three and
six, respectively. These upper bounds might not be tight. For example, for the \PR-distance the
bound might be closer to twice the agreement distance. This thought is also motivated by the fact
that the neighbourhoods of a network under \PR and the agreement distance are the same.

While the agreement distance is still NP-hard to compute, it avoids problems of shortest \PR- or
\SNPR-sequences as identified by \citet{KL18}. 
While for such a shortest sequence it might matter at which step of the sequence a reticulation
edge is added, an agreement graph has simply as many disagreement edges as needed. Furthermore,
while a sequence might traverse networks with more or less reticulations than the start and target
network, this is also irrelevant for agreement graphs.
Moreover, the \SNPR-distance between two networks of a certain class, for example of tree-child
networks, can differ if considered in the space of $\nets$ or just within this class, i.e., where
the \SNPR-sequence does not leave the class. This is by definition not the case for the agreement distance.
We therefore hope that it is easier to find exact and approximation algorithms for the agreement
distance than for the \PR-distance, just as it has been more fruitful to work with agreement forests
than with shortest \rSPR-sequences.

Beyond rooted binary phylogenetic networks it is interesting to see whether agreement graphs
and the agreement distance can be generalised to multifurcating phylogenetic networks or
even to directed graphs in general.
For unrooted phylogenetic trees, \citet{AS01} have shown that unrooted agreement
forests characterise the distance of the tree bisection and reconnection (TBR) operation. This
imposes the questions whether agreement graphs can also be defined for unrooted phylogenetic
networks and how they would relate to generalisations of the (unrooted) SPR and the TBR operation.

\acknowledgements
I thank the anonymous reviewers for their great comments and Simone Linz for helpful discussions.

\phantomsection


\begin{thebibliography}{19}
\providecommand{\natexlab}[1]{#1}
\providecommand{\url}[1]{\texttt{#1}}

\bibitem[Allen and Steel(2001)]{AS01}
B.~L. Allen and M.~Steel.
\newblock {Subtree Transfer Operations and Their Induced Metrics on Evolutionary Trees}.
\newblock \emph{Annals of Combinatorics}, 5\penalty0 (1):\penalty0 1--15, 2001.
\newblock \doi{10.1007/s00026-001-8006-8}.

\bibitem[Bonet and St.~John(2009)]{BStJ09}
M.~L. Bonet and K.~St.~John.
\newblock {Efficiently Calculating Evolutionary Tree Measures Using SAT}.
\newblock In O.~Kullmann, editor, \emph{Theory and Applications of
  Satisfiability Testing - SAT 2009}, pages 4--17. Springer Berlin Heidelberg, 2009.
\newblock \doi{10.1007/978-3-642-02777-2\_3}.

\bibitem[Bordewich and Semple(2005)]{BS05}
M.~Bordewich and C.~Semple.
\newblock {On the Computational Complexity of the Rooted Subtree Prune and Regraft Distance}.
\newblock \emph{Annals of Combinatorics}, 8\penalty0 (4):\penalty0 409--423, 2005.
\newblock \doi{10.1007/s00026-004-0229-z}.

\bibitem[Bordewich et~al.(2008)Bordewich, McCartin, and Semple]{BMS08}
M.~Bordewich, C.~McCartin, and C.~Semple.
\newblock A 3-approximation algorithm for the subtree distance between phylogenies.
\newblock \emph{Journal of Discrete Algorithms}, 6\penalty0 (3):\penalty0 458--471, 2008.
\newblock \doi{10.1016/j.jda.2007.10.002}.

\bibitem[Bordewich et~al.(2017{\natexlab{a}})Bordewich, Linz, and Semple]{BLS17}
M.~Bordewich, S.~Linz, and C.~Semple.
\newblock Lost in space? {G}eneralising subtree prune and regraft to spaces of phylogenetic networks.
\newblock \emph{Journal of Theoretical Biology}, 423:\penalty0 1--12, 2017{\natexlab{a}}.
\newblock \doi{10.1016/j.jtbi.2017.03.032}.

\bibitem[Bordewich et~al.(2017{\natexlab{b}})Bordewich, Scornavacca, Tokac, and Weller]{BSTW17}
M.~Bordewich, C.~Scornavacca, N.~Tokac, and M.~Weller.
\newblock On the fixed parameter tractability of agreement-based phylogenetic distances.
\newblock \emph{Journal of Mathematical Biology}, 74\penalty0 (1):\penalty0 239--257, 2017{\natexlab{b}}.
\newblock \doi{10.1007/s00285-016-1023-3}.

\bibitem[Chen et~al.(2015)Chen, Fan, and Sze]{CFS15}
J.~Chen, J.-H. Fan, and S.-H. Sze.
\newblock Parameterized and approximation algorithms for maximum agreement forest in multifurcating trees.
\newblock \emph{Theoretical Computer Science}, 562:\penalty0 496--512, 2015.
\newblock \doi{10.1016/j.tcs.2014.10.031}.

\bibitem[Dunn(2014)]{Dun14}
M.~Dunn.
\newblock Language phylogenies.
\newblock In C.~Bowern and B.~Evans, editors, \emph{The Routledge Handbook of Historical Linguistics}, chapter~7. Routledge, 2014.
\newblock \doi{10.4324/9781315794013}.

\bibitem[Felsenstein(2004)]{Fel04}
J.~Felsenstein.
\newblock \emph{Inferring phylogenies}, volume~2.
\newblock Sinauer Associates, 2004.

\bibitem[Francis et~al.(2018)Francis, Huber, Moulton, and Wu]{FHMW17}
A.~Francis, K.~T. Huber, V.~Moulton, and T.~Wu.
\newblock Bounds for phylogenetic network space metrics.
\newblock \emph{Journal of Mathematical Biology}, 76\penalty0 (5):\penalty0
  1229--1248, 2018.
\newblock \doi{10.1007/s00285-017-1171-0}.

\bibitem[Gambette et~al.(2017)Gambette, van Iersel, Jones, Lafond, Pardi, and Scornavacca]{GvIJLPS17}
P.~Gambette, L.~van Iersel, M.~Jones, M.~Lafond, F.~Pardi, and C.~Scornavacca.
\newblock Rearrangement moves on rooted phylogenetic networks.
\newblock \emph{PLOS Computational Biology}, 13\penalty0 (8):\penalty0 1--21, 2017.
\newblock \doi{10.1371/journal.pcbi.1005611}.

\bibitem[Huson et~al.(2010)Huson, Rupp, and Scornavacca]{HRS10}
D.~H. Huson, R.~Rupp, and C.~Scornavacca.
\newblock \emph{{Phylogenetic Networks: Concepts, Algorithms and Applications}}.
\newblock Cambridge University Press, 2010.

\bibitem[Janssen et~al.(2018)Janssen, Jones, Erd{\H{o}}s, van Iersel, and Scornavacca]{JJEvIS17}
R.~Janssen, M.~Jones, P.~L. Erd{\H{o}}s, L.~van Iersel, and C.~Scornavacca.
\newblock {Exploring the Tiers of Rooted Phylogenetic Network Space Using Tail Moves}.
\newblock \emph{Bulletin of Mathematical Biology}, 80\penalty0 (8):\penalty0
  2177--2208, 2018.
\newblock \doi{10.1007/s11538-018-0452-0}.

\bibitem[{Klawitter}(2018)]{Kla17}
J.~{Klawitter}.
\newblock The {SNPR} neighbourhood of tree-child networks.
\newblock \emph{Journal of Graph Algorithms and Applications}, 22\penalty0 (2):\penalty0 329--355, 2018.
\newblock \doi{10.7155/jgaa.00472}.

\bibitem[Klawitter and Linz(2019)]{KL18}
J.~Klawitter and S.~Linz.
\newblock {On the Subnet Prune and Regraft Distance}.
\newblock \emph{Electronic Journal of Combinatorics}, 26:\penalty0 329--355, 2019.
\newblock URL \url{www.combinatorics.org/ojs/index.php/eljc/article/view/v26i2p3}.

\bibitem[Semple and Steel(2003)]{SS03}
C.~Semple and M.~Steel.
\newblock \emph{Phylogenetics}, volume~24.
\newblock Oxford University Press on Demand, 2003.

\bibitem[St.~John(2017)]{StJ17}
K.~St.~John.
\newblock {Review Paper: The Shape of Phylogenetic Treespace}.
\newblock \emph{Systematic Biology}, 66\penalty0 (1):\penalty0 e83--e94, 2017.
\newblock \doi{10.1093/sysbio/syw025}.

\bibitem[Whidden et~al.(2013)Whidden, Beiko, and Zeh]{WBZ13}
C.~Whidden, R.~G. Beiko, and N.~Zeh.
\newblock {Fixed-Parameter Algorithms for Maximum Agreement Forests}.
\newblock \emph{SIAM Journal on Computing}, 42\penalty0 (4):\penalty0 1431--1466, 2013.
\newblock \doi{10.1137/110845045}.

\bibitem[Wu(2009)]{Wu09}
Y.~Wu.
\newblock A practical method for exact computation of subtree prune and regraft distance.
\newblock \emph{Bioinformatics}, 25\penalty0 (2):\penalty0 190--196, 2009.
\newblock \doi{10.1093/bioinformatics/btn606}.
\end{thebibliography}

\end{document}